\documentclass[12pt]{amsart}
\usepackage[T1]{fontenc}
\usepackage[utf8]{inputenc}
\usepackage{lmodern}
\usepackage{microtype}
\microtypesetup{expansion=false}
\usepackage{amsaddr}

\usepackage{amsmath,amssymb,amsthm,mathrsfs}

\usepackage{booktabs}
\usepackage{longtable}
\usepackage{tikz-cd}
\usepackage{placeins}
\usepackage{caption}
\usepackage{xcolor}
\usepackage{array}
\usepackage{tabularx}
\usepackage{etoolbox}
\newcolumntype{C}[1]{>{\centering\arraybackslash}p{#1}}
\newcolumntype{L}[1]{>{\raggedright\arraybackslash}p{#1}}
\newcolumntype{R}[1]{>{\raggedleft\arraybackslash}p{#1}}

\AtBeginEnvironment{tabular}{\footnotesize}
\AtBeginEnvironment{tabularx}{\footnotesize}
\AtBeginEnvironment{longtable}{\footnotesize}

\usepackage[colorlinks=true,
            linkcolor=red!50!black,
            citecolor=red!70!black,
            urlcolor=blue!80!black]{hyperref}
\usepackage[nameinlink,capitalise]{cleveref}

\numberwithin{equation}{section}

\newtheorem{thm}{Theorem}[section]
\newtheorem{lem}[thm]{Lemma}
\newtheorem{prop}[thm]{Proposition}
\newtheorem{cor}[thm]{Corollary}
\theoremstyle{definition}
\newtheorem{rem}[thm]{Remark}
\newtheorem{defn}[thm]{Definition}
\newtheorem{ex}[thm]{Example}
\theoremstyle{plain}

\DeclareMathOperator{\Aut}{Aut}
\DeclareMathOperator{\Orb}{Orb}
\DeclareMathOperator{\Stab}{Stab}
\DeclareMathOperator{\Sym}{Sym}
\DeclareMathOperator{\Hol}{Hol}
\DeclareMathOperator{\GL}{GL}
\DeclareMathOperator{\Cay}{Cay}
\DeclareMathOperator{\SCM}{SCM}
\DeclareMathOperator{\CB}{CB}
\DeclareMathOperator{\Conj}{Conj}
\DeclareMathOperator{\Orth}{Orth}
\DeclareMathOperator{\Perm}{Perm}

\newcommand{\F}{\mathbb{F}}
\newcommand{\Mat}{\mathcal{M}}

\newcommand{\Dp}{\Delta^{+}}
\newcommand{\Dm}{\Delta^{-}}
\newcommand{\Dc}{\Delta^{c}}

\newcommand{\Rt}[1]{R_{#1}}
\newcommand{\conjg}[1]{\gamma_{#1}}

\title{Colouring bijections of finite $3$-groups}

\author{Piotr Grzeszczuk}

\address{
Faculty of Computer Science\\
Bialystok University of Technology\\
Wiejska 45A\\
15-351 Białystok\\
Poland
}

\email{p.grzeszczuk@pb.edu.pl}

\thanks{This work was supported by Grant WZ/WI-IIT/2/2025 from the Bialystok University of Technology and funded by the Ministry of Science and Higher Education of Poland through resources allocated for research.}



\begin{document}

\begin{abstract}
Let $G$ be a finite group.  A bijection $\sigma\colon G\to G$ is a \emph{colouring
bijection} if the three maps
\[
\Dp\colon x\mapsto x\,\sigma(x),
\qquad
\Dm\colon x\mapsto x^{-1}\sigma(x),
\qquad
\Dc\colon x\mapsto \sigma(x)^{-1}x\,\sigma(x)
\]
are again bijections of $G$.  The first two conditions say that $\sigma$ is a strong
complete mapping.  The third is a genuinely nonabelian requirement.

Our main theorem is that every noncyclic $3$-group not isomorphic to the modular group
$M_{3^{r}}$ $(r\ge4)$ admits a colouring bijection.  This is the exact analogue, for
colouring bijections, of the theorem of Akhtar and Gagola on strong complete mappings.

The notion has three equivalent readings.  A colouring bijection properly colours the
Cayley graph $\mathscr{G}_3(G)=\Cay(G^{3},\mathbf S_3)$ with $|G|$ colours.  It
determines a triple of mutually orthogonal translation Latin squares based on $G$.  In
the orthomorphism graph of $G$ that triple is a triangle through the Cayley table.  It
also determines a common transversal of three arrays attached to $G$.  These are the
multiplication table, the division table, and the operation table of the conjugation
quandle.  The last of them is not a Latin square.

Two consequences follow.  Every noncyclic $3$-group $G\not\cong M_{3^{r}}$ $(r\ge4)$
carries three mutually orthogonal Latin squares of order $|G|$ based on $G$.  Moreover
$\chi(\mathscr{G}_3(G))=|G|$ for every such group.
\end{abstract}

\maketitle

\section*{Introduction}

Let $G$ be a finite group.  A bijection $\sigma\colon G\to G$ is a \emph{complete
mapping} if $\Dp_\sigma\colon x\mapsto x\sigma(x)$ is again a bijection.  It is an \emph{orthomorphism}
if $\Dm_\sigma\colon x\mapsto x^{-1}\sigma(x)$ is a bijection.  The two notions are equivalent under the
correspondence $\sigma\mapsto\Dp_\sigma$. Indeed,
$x^{-1}\Dp_\sigma(x)=\sigma(x)$, so $\sigma$ is a complete mapping if and only if
$\Dp_\sigma$ is an orthomorphism.

Complete mappings are the algebraic form of a combinatorial notion.  Let $L(r,c)=rc$ be
the Cayley table of $G$, a Latin square of order $|G|$, and call a set of $|G|$ cells, one
in each row, one in each column, carrying pairwise distinct symbols, a \emph{transversal}
of $L$.  The cells $\{(r,\sigma(r)):r\in G\}$ form a transversal exactly when $\sigma$ is
a complete mapping, and the translates $\{(r,\sigma(r)g)\}$, $g\in G$, then split the
table into $|G|$ disjoint transversals.  Hence $L$ has an orthogonal mate if and only if
$G$ admits a complete mapping (see \cite{E3} for a systematic account).

Which groups admit complete mappings is the content of the Hall--Paige theorem. Hall and Paige \cite{HP}
proved that a group with a nontrivial cyclic Sylow $2$-subgroup admits no complete
mapping. They conjectured the converse. Wilcox \cite{W} reduced the conjecture to the
simple groups. Evans \cite{E1} settled it for the sporadic groups, with the exception of
$J_4$. Bray completed the proof, see \cite{BCCSZ}. A finite group therefore admits a
complete mapping if and only if its Sylow $2$-subgroups are trivial or noncyclic.

A bijection that is at once a complete mapping and an orthomorphism is a \emph{strong
complete mapping}. The term was introduced by Hsu and Keedwell in 1984
\cite{HsuKeedwell1984,HsuKeedwell1985}. Earlier the same maps were called \emph{strong
orthomorphisms} \cite{Anderson1984} or \emph{strong permutations} \cite{Horton1990}. In
the context of block designs one also meets the name \emph{starter-adder pair}.

The existence question was
settled for abelian groups in \cite{E2}, with further constructions in \cite{BGZ}, and for
nonabelian $3$-groups and $2$-groups in \cite{AG} and \cite{ACG}, respectively.

\medskip

\noindent\textbf{The graphs $\mathscr{G}_m(G)$.}
For $m\ge 2$, Bagi\'{n}ski and Grzeszczuk \cite{BG} introduced the Cayley graphs
\[
\mathscr{G}_m(G)=\Cay(G^m,\mathbf S_m),
\]
whose connection set is built as follows.  For $\varnothing\neq A\subseteq\{1,\dots,m\}$
and $g\in G$ let $v_A(g)\in G^m$ be the vector with $i$-th coordinate $g$ for $i\in A$ and
$e$ otherwise, and call $A$ an \emph{interval} if $A=\{i,i+1,\dots,j\}$ for some
$i\le j$.  Then
\[
\mathbf S_m=\{\,v_A(g)\ :\ A\ \text{an interval},\ g\in G\setminus\{e\}\,\},
\]
so that $\mathbf S_m$ splits into $m(m+1)/2$ families, each of size $|G|-1$.  In
particular
\[
\mathbf S_2=\{(g,e),\ (e,g),\ (g,g)\},
\]
\[
\mathbf S_3=\{(g,e,e),\ (e,g,e),\ (e,e,g),\ (g,g,e),\ (e,g,g),\ (g,g,g)\},
\]
with $g$ ranging over $G\setminus\{e\}$. The vectors $(g,e,g)$ are absent because
$\{1,3\}$ is the unique nonempty subset of $\{1,2,3\}$ that is not an interval.  The
family $\mathscr{G}_m(G)$ arose in \cite{BG} in connection with a ring-theoretic problem,
where the motivation for its study is explained in detail.  These graphs determine $G$:
it is shown in \cite{BG} that $\mathscr{G}_m(G)\cong\mathscr{G}_m(H)$ if and only if
$G\cong H$.

\medskip

\noindent\textbf{From $m=2$ to $m=3$.}
The graph $\mathscr{G}_2(G)$ is the Latin square graph of the Cayley table. It is the
model for everything that follows. The family $(g,g)$ acts on $G^2$ with $|G|$ orbits. The
invariant of an orbit is $r^{-1}c$. Call $r^{-1}c$ the \emph{symbol} of the cell $(r,c)$.
This turns $G^2$ into the Latin square
\[
L^{-}(r,c)=r^{-1}c ,
\]
the left-division parastrophe of $L$. The three families of $\mathbf S_2$ then become the
three ways in which two cells of a Latin square can conflict: same column, same row, same
symbol. Rows, columns and symbol classes are cliques of size $|G|$. Hence
$\chi(\mathscr{G}_2(G))\ge|G|$. On the other hand, $c(x,y)=x^{-1}\tau(y)$ is a proper
colouring with $|G|$ colours precisely when $\tau$ and $y\mapsto y^{-1}\tau(y)$ are
bijections. That is, precisely when $\tau$ is an orthomorphism. So
$\chi(\mathscr{G}_2(G))=|G|$ if and only if $G$ admits an orthomorphism, equivalently a
complete mapping. By
Hall--Paige, this holds if and only if the Sylow $2$-subgroups of $G$ are trivial or
noncyclic.

The underlying mechanism is a count of the conditions imposed by each family of
generators. There are two axial families, $(g,e)$ and $(e,g)$. They force a single
property: bijectivity of $\tau$. There is one diagonal family, $(g,g)$. It adds exactly
one further condition, namely that $\tau$ is an orthomorphism.

For $m=3$ there are three axial families and three diagonal ones. The diagonal families
are $(g,g,e)$, $(e,g,g)$ and $(g,g,g)$. Each of them contributes one condition on a
colouring of the form $c(x,y,z)=x^{-1}\tau(y)z$. So the single extra condition of the
case $m=2$ becomes three conditions for $m=3$. These three conditions are the main theme
of this paper.

A strong complete mapping $\sigma$ is a \emph{colouring bijection}
if in addition
\[
\Dc(x)=\sigma(x)^{-1}x\,\sigma(x)
\]
is a bijection of $G$. A group admitting one is called \emph{colourable}.  

A colouring bijection $\sigma$ yields a proper colouring of $\mathscr{G}_3(G)$ with $|G|$
colours, namely $c(x,y,z)=x^{-1}\sigma^{-1}(y)z$. Hence
$\chi(\mathscr{G}_3(G))=|G|$ (Proposition~\ref{chromatic}). The name refers to this
colouring. The correspondence is exact. In the proof of Proposition~\ref{chromatic} the
neighbours of type $(gx,gy,z)$ force $\Dm$ to be injective. Those of type $(x,gy,gz)$
force $\Dp$, and those of type $(gx,gy,gz)$ force $\Dc$. The three defining conditions
therefore match the three diagonal families $(g,g,e)$, $(e,g,g)$ and $(g,g,g)$ one by
one. Running the same computations backwards shows that none of them can be dropped.
Finally, $\chi(\mathscr{G}_3(G))\ge|G|$ for every $G$. A colouring bijection is thus a
bijection that colours $\mathscr{G}_3(G)$ with the least possible number of colours.

Perhaps the most compelling motivation for the definition comes from the theory of orthogonal
Latin squares based on groups \cite{E3}. For a map $\theta\colon G\to G$ let
$L_\theta$ be the translation square $L_\theta(r,c)=c\,\theta(r)$.  A colouring bijection
$\sigma$ is exactly a triple of mutually orthogonal such squares of order $|G|$: the
Cayley table $L_{\mathrm{id}}$, the square $L_\sigma$, and the square $L_{\Dp}$ built
from $\Dp(x)=x\sigma(x)$ (Proposition~\ref{prop:mols}).  Two mutually orthogonal translation
squares correspond to a complete mapping, and so to the Hall--Paige theorem. Three of
them correspond to a colouring bijection, and the third square is not free but forced,
being the pointwise product of the other two.  Combined with the main theorem, this
yields three mutually orthogonal Latin squares of order $|G|$ based on $G$ for every
noncyclic $3$-group $G\not\cong M_{3^{r}}$.  Section~\ref{sec:mols} develops this in the
language of the orthomorphism graph of $G$.

The third condition is not, then, an ad hoc strengthening, and it has a second reading as
a statement about transversals.  The maps $\Dp$, $\Dm$, $\Dc$ are the diagonals along
$\sigma$ of the three binary operations $x\cdot y$, $x^{-1}y$ and
$x\triangleright y=y^{-1}xy$.  The first two are the Cayley table and its division
parastrophe; the third is the operation table of the \emph{conjugation quandle} of $G$.
Thus $\sigma$ is a colouring bijection if and only if $\{(x,\sigma(x)):x\in G\}$ is a
common transversal of these three arrays.  The last of them is not a Latin square and is
not a parastrophe of the Cayley table, which is why colouring bijections lie outside the
framework in which complete mappings live.  Section~\ref{sec:quandle} develops this.

\medskip

\noindent\textbf{Known results and the main theorem.}
Evans \cite{E2} determined which abelian groups admit strong complete mappings (further
constructions are in \cite{BGZ}).  For nonabelian $3$-groups the existence problem was
settled by Akhtar and Gagola \cite{AG}, with one family of exceptions: the modular
$3$-groups
\[
M_{3^{r}}=\langle a,b\mid a^{3^{r-1}}=b^3=e,\ bab^{-1}=a^{1+3^{r-2}}\rangle,
\qquad r\ge 4 .
\]
Here $M_{3^{r}}$ denotes, as usual, the modular group of order $3^{r}$, that is, the
unique nonabelian $3$-group of order $3^{r}$ with a cyclic maximal subgroup.  Note that $M_{3^{r}}$ has nilpotency
class $2$ for every $r\ge3$, since $[a,b]=a^{3^{r-2}}$ is central.
They proved that every noncyclic $3$-group not isomorphic to $M_{3^{r}}$ $(r\ge4)$ admits a
strong complete mapping.  Our main result is the exact analogue for colouring bijections.

The natural framework for this existence problem is Evans' theory of orthomorphism graphs
\cite{E3,E4}. Section~\ref{sec:mols} recalls it in the form used here. The vertices of
the translation graph $\Perm(G)$ are the Latin squares based on $G$, and adjacency is
orthogonality. The orthomorphism graph $\Orth(G)$ is the neighbourhood of the Cayley
table. The inversion map $\iota(x)=x^{-1}$ is a vertex of $\Perm(G)$. An orthomorphism of
$G$ is a strong complete mapping precisely when it is orthogonal to $\iota$
\cite[\S2]{E4}. Thus $G$ admits a strong complete mapping if and only if $\mathrm{id}$
and $\iota$ have a common neighbour in $\Perm(G)$. For $|G|$ odd these two vertices are
themselves adjacent. The condition then reads that the edge $\{\mathrm{id},\iota\}$ lies
in a triangle. For $|G|$ even they are never adjacent.

For colouring bijections the third vertex of the triangle is not free. By
Proposition~\ref{prop:mols}, a bijection $\sigma$ is a colouring bijection precisely when
the triangle on $\mathrm{id}$ and $\sigma$ closes on the vertex $\Dp_\sigma$ determined
by $\sigma$ itself. In this language the theorem of Akhtar and Gagola produces the first
triangle for every noncyclic $3$-group not isomorphic to $M_{3^{r}}$ $(r\ge4)$. The Main
Theorem below produces the second one in the same range. Both statements leave open the
groups $M_{3^{r}}$, $r\ge4$ (see the concluding remarks and the survey \cite{E5}).

\begin{quote}
\textbf{Main Theorem.}
\emph{Every noncyclic $3$-group not isomorphic to $M_{3^{r}}$ $(r\ge4)$ admits a colouring
bijection.  Consequently $\chi(\mathscr{G}_3(G))=|G|$ for every such group.}
\end{quote}

The base cases are the two nonabelian groups of order $27$, for which colouring
bijections are produced by computer search.  The inductive step lifts a colouring
bijection of $G/H$ to one of $G$ along a normal subgroup $H$ isomorphic to
$C_3\times C_3$ or to $C_9\times C_3$. The two cases are joined by a structural lemma:
every noncyclic $3$-group of order greater than $27$ whose maximal subgroups are all
noncyclic contains a normal subgroup isomorphic to $C_3\times C_3$ with noncyclic
quotient.

\medskip

\noindent\textbf{Organisation.}
Sections~\ref{Section1}--\ref{sec:hol} introduce the colouring bijection and justify its
study. Section~\ref{Section1} fixes the notation, proves the graph-colouring
interpretation (Proposition~\ref{chromatic}) and records the elementary existence
results. Section~\ref{sec:mols} identifies colouring bijections with distinguished
triples of mutually orthogonal Latin squares based on $G$ (Proposition~\ref{prop:mols}).
Section~\ref{sec:quandle} identifies them with common transversals of three arrays: the
multiplication table, its division parastrophe, and the operation table of the
conjugation quandle. Section~\ref{sec:hol} describes the action of the holomorph
$\Hol(G)=G\rtimes\Aut(G)$, which enlarges each colouring bijection to a whole orbit.

The remaining sections address the existence problem in the class of finite $3$-groups.
Section~\ref{sec:order27} settles the groups of order at most $27$.
Section~\ref{sec:lift} develops the lifting machinery and proves the two lifting theorems.
Section~\ref{sec:Main} proves the main theorem. Section~\ref{sec:omega} draws a
consequence for the clique number of the orthomorphism graph. Every group of exponent $3$
and order greater than $3$ is colourable. It therefore carries five mutually orthogonal
Latin squares and satisfies $\omega(\Orth(G))\ge4$. For every order greater than $27$
this improves to $\omega(\Orth(G))\ge7$. For the extraspecial groups $3^{1+2m}$ of either
exponent $3$ or exponent $9$ there is the bound $\omega(\Orth(G))\ge3^{m}-2$, which is
informative for $m\ge2$. The extraspecial group $H_3$ of order
$27$ is the only case in which the first bound is the best one available. The appendix
lists the auxiliary maps of $C_9\times C_3$ used in the proof of the second lifting
result, Theorem~\ref{thm:lift-C9xC3}. They were found by computer search.

\section{Colouring bijections}\label{Section1}

All groups are finite and written multiplicatively; $e$ denotes the
identity, $Z(G)$ the centre of $G$, and
\[
\rho_x\colon G\to G,\qquad \rho_x(g)=xgx^{-1},
\]
the inner automorphism determined by $x\in G$.  If $H\trianglelefteq G$,
then $\rho_x$ restricts to an automorphism of $H$, still denoted by
$\rho_x$, and $\pi\colon G\to G/H$ is the natural projection.

\begin{defn}\label{def:scm_bc}
A bijection $\sigma\colon G\to G$ is a \emph{strong complete mapping}
(SCM) if the two maps
\[
  \Dp(x)=x\cdot\sigma(x)
  \qquad\text{and}\qquad
  \Dm(x)=x^{-1}\cdot \sigma(x)
\]
are both bijections.  An SCM $\sigma$ is called a \emph{colouring
bijection} (CB) if additionally the map
\[
  \Dc(x)=\sigma(x)^{-1}\cdot x\cdot \sigma(x)
\]
is a bijection.  We write $\SCM(G)$ and $\CB(G)$ for the respective sets
of all strong complete mappings and all colouring bijections of $G$.
\end{defn}

A group $G$ is \emph{strongly admissible} if $\SCM(G)\neq\varnothing$, and
\emph{colourable} if $\CB(G)\neq\varnothing$.  By
Definition~\ref{def:scm_bc},
\begin{equation}\label{eq:CB-in-SCM}
\CB(G)\subseteq\SCM(G),
\end{equation}
so every colourable group is strongly admissible.  When $\sigma$ has to
be displayed we write $\Dp_\sigma,\Dm_\sigma,\Dc_\sigma$.

\begin{rem}\label{rem:abelian}
If $G$ is abelian, then $\Dc_\sigma(x)=\sigma(x)^{-1}x\,\sigma(x)=x$ for
every $x$, so $\Dc_\sigma=\mathrm{id}_G$ is a bijection for every
$\sigma$.  Hence
\[
\CB(G)=\SCM(G)\qquad\text{for every abelian group }G .
\]
The condition imposed on $\Dc$ is therefore a genuinely nonabelian
requirement, and by \eqref{eq:CB-in-SCM} it is a restriction of the
notion of a strong complete mapping.
\end{rem}

\medskip

The following proposition is the combinatorial motivation for
Definition~\ref{def:scm_bc}.  We keep the notation of the Introduction for
the graph $\mathscr{G}_3(G)=\Cay(G^{3},\mathbf S_3)$ of \cite{BG} and use
that it contains cliques of size $|G|$, so that
$\chi(\mathscr{G}_3(G))\ge |G|$.

\begin{prop}\label{chromatic}
Let $G$ be a finite group and let $\sigma\in\CB(G)$.  Then
\[
c\colon G^{3}\longrightarrow G,
\qquad
c(x,y,z)=x^{-1}\,\sigma^{-1}(y)\,z,
\]
is a proper colouring of $\mathscr{G}_3(G)$ with $|G|$ colours.
Consequently, every colourable group $G$ satisfies
\[
\chi(\mathscr{G}_3(G))=|G| .
\]
\end{prop}

\begin{proof}
Every neighbour of a vertex $(x,y,z)$ of $\mathscr{G}_3(G)$ has one of the
six forms
\[
(gx,y,z),\quad (x,gy,z),\quad (x,y,gz),\quad
(gx,gy,z),\quad (x,gy,gz),\quad (gx,gy,gz),
\]
where $g\in G\setminus\{e\}$.  Fix such a $g$ and put
\[
u=\sigma^{-1}(y),
\qquad
v=\sigma^{-1}(gy),
\]
so that
\begin{equation}\label{eq:uv}
\sigma(u)=y,
\qquad
\sigma(v)=g\,\sigma(u),
\qquad
g=\sigma(v)\,\sigma(u)^{-1}.
\end{equation}
Since $g\neq e$ and $\sigma$ is injective, we have $v\neq u$.  In this
notation
\[
c(x,y,z)=x^{-1}u\,z,
\qquad
c(x,gy,z)=x^{-1}v\,z .
\]
We show that in each of the six cases the assumption that the neighbour
receives the colour $c(x,y,z)$ leads to a contradiction.

For a neighbour of the form $(gx,y,z)$ the colour equals $x^{-1}g^{-1}u\,z$.
If this equals $x^{-1}u\,z$, then $g^{-1}u=u$, whence $g=e$, a
contradiction.  Likewise, the colour of $(x,y,gz)$ equals $x^{-1}u\,gz$.
If this equals $x^{-1}u\,z$, then $gz=z$ and again $g=e$.  Finally, the
colour of $(x,gy,z)$ equals $x^{-1}v\,z$, and if this equals $x^{-1}u\,z$,
then $v=u$, contrary to $v\neq u$.

Consider now a neighbour of the form $(gx,gy,z)$, whose colour is
$x^{-1}g^{-1}v\,z$.  Assume that
\[
x^{-1}g^{-1}v\,z=x^{-1}u\,z .
\]
Cancelling $x^{-1}$ on the left and $z$ on the right gives $g^{-1}v=u$.
Substituting $g^{-1}=\sigma(u)\,\sigma(v)^{-1}$ from \eqref{eq:uv} we obtain
$\sigma(u)\,\sigma(v)^{-1}v=u$, that is,
\[
\sigma(v)^{-1}v=\sigma(u)^{-1}u .
\]
Inverting both sides yields $v^{-1}\sigma(v)=u^{-1}\sigma(u)$, i.e.\ 
$\Dm_\sigma(v)=\Dm_\sigma(u)$.  As $\sigma\in\CB(G)$, the map $\Dm_\sigma$
is injective, so $v=u$, a contradiction.

Next let the neighbour be $(x,gy,gz)$, with colour $x^{-1}v\,gz$, and
assume that
\[
x^{-1}v\,gz=x^{-1}u\,z .
\]
Cancelling as before gives $v\,g=u$, and substituting
$g=\sigma(v)\,\sigma(u)^{-1}$ we get $v\,\sigma(v)\,\sigma(u)^{-1}=u$, that
is,
\[
v\,\sigma(v)=u\,\sigma(u),
\]
i.e.\ $\Dp_\sigma(v)=\Dp_\sigma(u)$.  Since $\Dp_\sigma$ is injective, again
$v=u$, a contradiction.

Finally, let the neighbour be $(gx,gy,gz)$, with colour
$x^{-1}g^{-1}v\,gz$, and assume that
\[
x^{-1}g^{-1}v\,gz=x^{-1}u\,z .
\]
Cancelling gives $g^{-1}v\,g=u$, and substituting for $g$ and $g^{-1}$
from \eqref{eq:uv} we obtain
$\sigma(u)\,\sigma(v)^{-1}\,v\,\sigma(v)\,\sigma(u)^{-1}=u$, that is,
\[
\sigma(v)^{-1}v\,\sigma(v)=\sigma(u)^{-1}u\,\sigma(u),
\]
i.e.\ $\Dc_\sigma(v)=\Dc_\sigma(u)$.  The map $\Dc_\sigma$ is injective, so
$v=u$, once more a contradiction.

Hence adjacent vertices always receive distinct colours, and $c$ is a
proper colouring.  It is surjective, since $c(e,\sigma(w),e)=w$ for every
$w\in G$, so it uses exactly $|G|$ colours.  Together with
$\chi(\mathscr{G}_3(G))\ge|G|$ this gives $\chi(\mathscr{G}_3(G))=|G|$.
\end{proof}

Colourability is a restriction only at the primes $2$ and $3$.

\begin{cor}\label{cor:coprime6}
If $\gcd(|G|,6)=1$, then $\sigma(x)=x^{2}$ is a colouring bijection of
$G$.  In particular, every group of order coprime to $6$ is colourable.
\end{cor}

\begin{proof}
As $2\nmid|G|$, the map $x\mapsto x^{2}$ is a bijection of $G$.  Moreover
\[
\Dp(x)=x\cdot x^{2}=x^{3},
\qquad
\Dm(x)=x^{-1}\cdot x^{2}=x,
\qquad
\Dc(x)=x^{-2}\cdot x\cdot x^{2}=x,
\]
and $x\mapsto x^{3}$ is a bijection because $3\nmid|G|$.
\end{proof}

We shall use the following theorem of Evans.

\begin{thm}[{\cite{E2}}]\label{3-group-abelian}
Every noncyclic abelian $3$-group is strongly admissible, and hence, by
Remark~\ref{rem:abelian}, colourable.
\end{thm}

\begin{rem}\label{rem:cyclic-excluded}
Noncyclicity cannot be dropped: a nontrivial cyclic $3$-group admits no
strong complete mapping \cite{E2,E3}, hence by \eqref{eq:CB-in-SCM} it is
not colourable.  This explains the exclusion of cyclic groups in all
statements below.
\end{rem}

\section{Colouring bijections and mutually orthogonal Latin squares}\label{sec:mols}

The three conditions of Definition~\ref{def:scm_bc} have a classical reading. It comes
from the theory of orthogonal Latin squares based on a group. The standard reference is
Evans \cite{E3}. We first recall that framework in the form we need: translation squares,
orthogonality, and the graph $\Perm(G)$ whose vertices are \emph{all} of them, of which
the orthomorphism graph $\Orth(G)$ is the neighbourhood of the Cayley table. We then show
that colouring bijections are exactly the triangles of $\Perm(G)$ through the Cayley
table (Proposition~\ref{prop:mols}), that each of them carries a second triangle disjoint
from the first (Proposition~\ref{prop:bowtie}), and that for $|G|$ odd the two triangles
close up into a wheel on five vertices (Proposition~\ref{prop:odd-triangles}). 

\subsection{The translation graph and the orthomorphism graph}\label{subsec:orth-graph}

Throughout, $\theta,\varphi$ denote maps $G\to G$. Recall the notion of orthogonality.
Let $A,B$ be two Latin squares of order $n$ on the same symbol set. They are
\emph{orthogonal}, written $A\perp B$, if the map
$(r,c)\mapsto\bigl(A(r,c),B(r,c)\bigr)$ is a bijection. Its target is the set of ordered
pairs of symbols. A family of squares is \emph{mutually orthogonal} (MOLS) if its members
are pairwise orthogonal.

The squares attached to $G$ are the following. Let $\theta\colon G\to G$ be a map. Let
$L_\theta$ be the array on $G\times G$ given by
\[
L_\theta(r,c)=c\cdot\theta(r).
\]
We call it the \emph{translation square} of $\theta$ \emph{based on} $G$. Take
$\theta=\mathrm{id}$. This gives $L_{\mathrm{id}}(r,c)=cr$. That is the Cayley table of
$G$, with the roles of the two arguments interchanged. Two computations govern everything
that follows.

\smallskip
\noindent\textbf{(T1)}\quad\emph{$L_\theta$ is a Latin square if and only if $\theta$ is
a bijection of $G$.}

\noindent Indeed, for fixed $r$ the row $c\mapsto c\,\theta(r)$ is a permutation of $G$
whatever $\theta$ is, whereas for fixed $c$ the column $r\mapsto c\,\theta(r)$ is a
permutation of $G$ exactly when $\theta$ is one.  In particular $L_{\mathrm{id}}$ is
always a Latin square.

\smallskip
\noindent\textbf{(T2)}\quad\emph{For bijections $\theta,\varphi$ of $G$,}
\[
L_\theta\perp L_\varphi
\qquad\Longleftrightarrow\qquad
r\mapsto \theta(r)^{-1}\varphi(r)\ \text{ is a bijection of } G .
\]

\noindent Indeed, fix $(u,v)\in G\times G$ and solve $L_\theta(r,c)=u$,
$L_\varphi(r,c)=v$.  The first equation gives $c=u\,\theta(r)^{-1}$, and substituting it
into the second one leaves
\[
\theta(r)^{-1}\varphi(r)=u^{-1}v .
\]
As $(u,v)$ runs over $G\times G$ the element $u^{-1}v$ runs over $G$, so the pair map is
a bijection precisely when $r\mapsto\theta(r)^{-1}\varphi(r)$ attains every value exactly
once.

\smallskip
By (T1) the translation squares based on $G$ that are Latin are indexed by the bijections
of $G$, and by (T2) their orthogonality is a relation between such bijections.  This makes
the following graph, rather than any subgraph of it, the natural object.

Two products occur side by side in what follows, and we keep them apart notationally:
juxtaposition or $\circ$ always denotes composition of maps,
$(\varphi\circ\theta)(r)=\varphi(\theta(r))$, whereas a product taken in $G$ is always
written with its argument displayed, as in $r\mapsto\theta(r)^{-1}\varphi(r)$.  Thus
$\Sym(G)$ is a group under composition, while an expression such as
$\theta(r)^{-1}\varphi(r)$ is an element of $G$.

\begin{defn}\label{def:perm-graph}
For bijections $\theta,\varphi$ of $G$ let
\[
w(\theta,\varphi)\colon G\to G,
\qquad
w(\theta,\varphi)(r)=\theta(r)^{-1}\varphi(r),
\]
be the \emph{witness} of the pair.  The \emph{translation graph} $\Perm(G)$ has as
vertices all bijections $\theta\colon G\to G$, with an edge $\theta\sim\varphi$ whenever
$w(\theta,\varphi)$ is a bijection.
\end{defn}

The witness of an edge is itself a vertex of $\Perm(G)$, and $w(\varphi,\theta)$ is the
pointwise inverse of $w(\theta,\varphi)$, so the relation is symmetric.  Witnesses will
label every edge we draw, and the three maps of Definition~\ref{def:scm_bc} will turn out
to be witnesses of three particular edges.

By (T1) and (T2) the vertices of $\Perm(G)$ are exactly the Latin translation squares
based on $G$, its edges are exactly the orthogonal pairs among them,
\[
\theta\ \text{a vertex}\iff L_\theta\ \text{is Latin},
\qquad
\theta\sim\varphi\iff L_\theta\perp L_\varphi ,
\]
and the Cayley table is the vertex $\mathrm{id}$.  Recall that a bijection $\theta$ is an
\emph{orthomorphism} of $G$ if $r\mapsto r^{-1}\theta(r)=w(\mathrm{id},\theta)(r)$ is
again a bijection. By (T2) with $\theta=\mathrm{id}$ this says precisely that
$L_\theta$ is orthogonal to the Cayley table, that is, $\theta\sim\mathrm{id}$.  The classical object of \cite{E3} is therefore a link in $\Perm(G)$.

\begin{defn}\label{def:orth-graph}
The \emph{orthomorphism graph} $\Orth(G)$ is the subgraph of $\Perm(G)$ induced on the
neighbourhood of the vertex $\mathrm{id}$. Its vertices are the orthomorphisms of $G$ and
$\theta\sim\varphi$ has the same meaning as in Definition~\ref{def:perm-graph}.
\end{defn}

We allow ourselves to write $\theta\in\Orth(G)$ and $|\Orth(G)|$ when we mean the vertex
set of $\Orth(G)$, that is, the set of orthomorphisms of $G$.  Note that for $|G|>1$ the
map $\mathrm{id}$ is not an orthomorphism, since $r\mapsto r^{-1}r$ is constant. The
Cayley table is the \emph{anchor} of $\Orth(G)$ and not one of its vertices.  (For the
trivial group, $\mathrm{id}$ is an orthomorphism and $\Perm(G)$ is a single vertex with a
loop; we exclude this case throughout.)  The passage between the two graphs is governed by the following lemma, which
also shows that no information is lost: all links of $\Perm(G)$ are copies of $\Orth(G)$.

\begin{lem}\label{lem:perm-cayley}
Let $G$ be a finite group with $|G|>1$.
\begin{enumerate}
\item For bijections $\theta,\varphi$ of $G$ one has
$\theta\sim\varphi$ if and only if $\varphi\circ\theta^{-1}\in\Orth(G)$.
\item The set $\Orth(G)$ is closed under inversion in $\Sym(G)$ and does not contain
$\mathrm{id}$, so it is a legitimate connection set, and by (1)
\[
\Perm(G)=\Cay\bigl(\Sym(G),\Orth(G)\bigr).
\]
In particular $\Perm(G)$ is vertex-transitive, has $|G|!$ vertices and is regular of
degree $|\Orth(G)|$.
\item For every $\psi\in\Sym(G)$ the map $\theta\mapsto\theta\circ\psi$ is an
automorphism of $\Perm(G)$. Consequently the neighbourhood of any vertex $\theta$ induces
a graph isomorphic to $\Orth(G)$.
\end{enumerate}
\end{lem}

\begin{proof}
(1) Substituting $s=\theta(r)$ gives the identity
\[
w(\theta,\varphi)=w\bigl(\mathrm{id},\varphi\circ\theta^{-1}\bigr)\circ\theta ,
\]
so the two witnesses differ by composition with the bijection $\theta$.  One is therefore
a bijection precisely when the other is, and the latter condition says exactly that
$\varphi\circ\theta^{-1}$ is an orthomorphism.

(2) Let $\chi\in\Orth(G)$ and put $s=\chi(r)$.  Then
$s^{-1}\chi^{-1}(s)=\chi(r)^{-1}r=\bigl(r^{-1}\chi(r)\bigr)^{-1}$, and as $r$ runs over
$G$ so does $s$. Being the composite of the bijection $r\mapsto r^{-1}\chi(r)$ with
inversion and with $\chi^{-1}$, the map $s\mapsto s^{-1}\chi^{-1}(s)$ is a bijection, so
$\chi^{-1}\in\Orth(G)$.  Moreover $\mathrm{id}\notin\Orth(G)$ because $|G|>1$.  Thus the
right-hand side is a simple undirected graph, and by (1) its adjacency is that of
$\Perm(G)$.

(3) Right composition $\theta\mapsto\theta\circ\psi$ preserves adjacency for every
$\psi\in\Sym(G)$: writing $D(r)=\theta(r)^{-1}\varphi(r)$ for the map governing
$\theta\sim\varphi$, the map governing the pair $\theta\circ\psi$, $\varphi\circ\psi$ is
\[
r\longmapsto\theta\bigl(\psi(r)\bigr)^{-1}\varphi\bigl(\psi(r)\bigr)=(D\circ\psi)(r),
\]
which is a bijection precisely when $D$ is one.  Taking $\psi=\theta^{-1}$ carries
$\theta$ to $\mathrm{id}$, hence the neighbourhood of $\theta$ onto that of $\mathrm{id}$.
\end{proof}

\begin{lem}\label{lem:cliques-mols}
Let $G$ be a finite group and $k\ge1$.
\begin{enumerate}
\item The $k$-cliques of $\Perm(G)$ are exactly the sets of $k$ mutually orthogonal
Latin squares based on $G$.
\item Every $k$-clique of $\Perm(G)$ can be carried by an automorphism of $\Perm(G)$ to a
$k$-clique containing $\mathrm{id}$, whose remaining $k-1$ vertices form a $(k-1)$-clique
of $\Orth(G)$.  In particular
\[
\omega\bigl(\Perm(G)\bigr)=\omega\bigl(\Orth(G)\bigr)+1 .
\]
\end{enumerate}
\end{lem}

\begin{proof}
(1) is (T1) together with (T2), and (2) follows from Lemma~\ref{lem:perm-cayley}(3).
\end{proof}

Part (2) is the usual normalisation: a set of MOLS based on $G$ may always be moved so
that one of its members is the Cayley table, which is why the classical theory works with
$\Orth(G)$ and counts an $r$-clique there as $r+1$ MOLS \cite[Lemma~2]{E4}.  In the same
vein, adjoining the constant row $e$ to the $k$ rows given by a $k$-clique produces a
$(|G|,k+1;1)$-difference matrix over $G$, and every such matrix arises in this way after
multiplying all of its rows pointwise by the inverse of one of them. Normalising a second
row is then exactly the passage to $\Orth(G)$ \cite{Jun}, \cite[\S2]{E4}.  Finally, $\Perm(G)$ has an edge if
and only if $\Orth(G)\neq\varnothing$, that is, if and only if $G$ admits a complete
mapping. By Hall--Paige \cite{HP,W,E1,BCCSZ} this fails exactly when the Sylow
$2$-subgroup of $G$ is nontrivial and cyclic, and then $\Perm(G)$ has no edges at all.

\subsection{Colouring bijections are the canonical triangles}

\begin{prop}\label{prop:mols}
Let $\sigma\colon G\to G$ be a map and put $\Dp(r)=r\,\sigma(r)$.  The following are
equivalent.
\begin{enumerate}
\item $\sigma\in\CB(G)$;
\item $L_{\mathrm{id}}$, $L_{\sigma}$ and $L_{\Dp}$ are three mutually orthogonal Latin
      squares of order $|G|$.
\end{enumerate}
\end{prop}

\begin{proof}
By (T1) the array $L_{\mathrm{id}}$ is always Latin, while $L_\sigma$ and $L_{\Dp}$ are
Latin if and only if $\sigma$ and $\Dp$ are bijections.  By (T2) the three possible
orthogonalities are governed by the three maps
\[
\begin{aligned}
L_{\mathrm{id}}\perp L_\sigma
   &:\ r\mapsto \mathrm{id}(r)^{-1}\sigma(r)=r^{-1}\sigma(r)=\Dm(r),\\
L_{\mathrm{id}}\perp L_{\Dp}
   &:\ r\mapsto \mathrm{id}(r)^{-1}\Dp(r)=r^{-1}r\,\sigma(r)=\sigma(r),\\
L_\sigma\perp L_{\Dp}
   &:\ r\mapsto \sigma(r)^{-1}\Dp(r)=\sigma(r)^{-1}r\,\sigma(r)=\Dc(r).
\end{aligned}
\]
Hence statement~(2) holds if and only if the four maps $\sigma$, $\Dp$, $\Dm$ and $\Dc$
are bijections of $G$, which by Definition~\ref{def:scm_bc} is statement~(1).  (The
second line imposes no condition beyond the bijectivity of $\sigma$, which is already
required for $L_\sigma$ to be Latin.)
\end{proof}

\medskip

In the language of Subsection~\ref{subsec:orth-graph}, Proposition~\ref{prop:mols} says
that
\[
\sigma\in\CB(G)
\qquad\Longleftrightarrow\qquad
\{\mathrm{id},\sigma,\Dp\}\ \text{is a triangle of}\ \Perm(G),
\]
that is, a triangle through the Cayley table.  The three defining conditions of a
colouring bijection are distributed over its three edges: $\mathrm{id}\sim\sigma$ holds
exactly when $\Dm$ is a bijection, $\mathrm{id}\sim\Dp$ exactly when $\Dp$ is one, and
$\sigma\sim\Dp$ exactly when $\Dc$ is one.  Equivalently, by
Lemma~\ref{lem:cliques-mols}, $\sigma$ and $\Dp$ are adjacent vertices of $\Orth(G)$.

What makes the resulting triangle $L_{\mathrm{id}}$--$L_\sigma$--$L_{\Dp}$ \emph{canonical}
is that its third square is not a free vertex but the forced one
$\Dp(r)=r\,\sigma(r)$, the pointwise product of the other two.  Beyond ``$\sigma$ is an
orthomorphism'' the sole new demand is the edge $\sigma\sim\Dp$, that is, $\Dc$.  For
abelian $G$ that edge comes for free by Remark~\ref{rem:abelian}, and one recovers the
classical triple $\{\mathrm{id},\sigma,\sigma+\mathrm{id}\}$ attached to a strong complete
mapping.
\medskip

Consequently, combining Theorem~\ref{thm:main} with Proposition~\ref{prop:mols} gives:

\begin{cor}\label{cor:mols-3groups}
Let $G$ be a noncyclic $3$-group not isomorphic to $M_{3^{r}}$ $(r\ge4)$.  Then there exist
three mutually orthogonal Latin squares of order $|G|$ based on $G$, namely
$L_{\mathrm{id}}$, $L_{\sigma}$ and $L_{\Dp_\sigma}$ for any $\sigma\in\CB(G)$.
\end{cor}

\medskip

\begin{cor}\label{cor:cb-witness}
Let $G$ be any finite group.  For a bijection $\sigma$ of $G$ the three maps of
Definition~\ref{def:scm_bc} are the witnesses
\[
w(\mathrm{id},\sigma)=\Dm,\qquad w(\iota,\sigma)=\Dp,\qquad w(\sigma,\Dp)=\Dc,
\]
and
\[
\sigma\in\SCM(G)\iff \sigma\sim\mathrm{id}\ \text{and}\ \sigma\sim\iota,
\qquad
\sigma\in\CB(G)\iff \sigma\sim\mathrm{id},\ \sigma\sim\iota\ \text{and}\
\sigma\sim w(\iota,\sigma).
\]
\end{cor}

\begin{proof}
The three witnesses are the computations $w(\mathrm{id},\sigma)(r)=r^{-1}\sigma(r)$,
$w(\iota,\sigma)(r)=r\,\sigma(r)$ and $w(\sigma,\Dp)(r)=\sigma(r)^{-1}r\,\sigma(r)$, none
of which uses the order of $G$.  By definition $\sigma\sim\mathrm{id}$ means that
$\Dm$ is a bijection and $\sigma\sim\iota$ that $\Dp$ is one, which together with the
bijectivity of $\sigma$ is Definition~\ref{def:scm_bc} for a strong complete mapping.
Given these, $w(\iota,\sigma)=\Dp$ and $\sigma\sim\Dp$ means that $\Dc$ is a bijection.
\end{proof}

Thus a colouring bijection is a vertex adjacent to the Cayley table and to $\iota$, and
adjacent as well to the witness of its own edge with $\iota$: a closure condition inside
$\Perm(G)$ that cannot even be stated in $\Orth(G)$, since $\Dp$ and $\Dm$ need not be
vertices there.

\begin{prop}\label{prop:bowtie}
Let $G$ be a finite group and let $\sigma$ be a bijection of $G$.  Then
$\sigma\in\CB(G)$ if and only if
\[
\{\mathrm{id},\sigma,\Dp\}
\qquad\text{and}\qquad
\{\iota,\sigma,\Dm\}
\]
are triangles of $\Perm(G)$.  Their six edges are witnessed by
\[
\begin{aligned}
w(\mathrm{id},\sigma)&=\Dm, &\qquad w(\mathrm{id},\Dp)&=\sigma, &\qquad w(\sigma,\Dp)&=\Dc,\\
w(\iota,\sigma)&=\Dp, &\qquad w(\iota,\Dm)&=\sigma, &\qquad w(\sigma,\Dm)&=(\Dc)^{-1},
\end{aligned}
\]
and only three of them carry information: given that $\Dp$ and $\Dm$ are bijections, the
edges $\mathrm{id}\sim\Dp$, $\iota\sim\Dm$ and $\sigma\sim\Dm$ hold automatically.  Every
colouring bijection therefore produces two triples of mutually orthogonal Latin squares
based on $G$, one containing the Cayley table and one not.  If $3$ divides $|G|$, the two
triangles meet in $\sigma$ alone.  If $|G|$ is odd, the further edges
$\mathrm{id}\sim\iota$ and $\Dp\sim\Dm$ are present and the bowtie closes into the wheel
of Proposition~\ref{prop:odd-triangles} below. If $|G|$ is even, $\mathrm{id}\sim\iota$ is never
an edge, since $r\mapsto r^{-2}$ is then not injective.  If $\exp(G)\le2$ the picture
collapses altogether, since then $\iota=\mathrm{id}$ and $\Dp=\Dm$.
\end{prop}

\begin{proof}
Suppose first that $\sigma\in\CB(G)$.  The six witnesses displayed above are computed directly, the parity of $|G|$ playing no
role:
$r^{-1}\sigma(r)=\Dm(r)$, $r^{-1}r\sigma(r)=\sigma(r)$,
$\sigma(r)^{-1}r\sigma(r)=\Dc(r)$ and, using $\iota(r)^{-1}=r$, $r\sigma(r)=\Dp(r)$,
$r\,r^{-1}\sigma(r)=\sigma(r)$, $\sigma(r)^{-1}r^{-1}\sigma(r)=\Dc(r)^{-1}$.  All six are
bijections, so both triples are triangles.

Conversely, suppose that both triples are triangles.  Then $\mathrm{id}\sim\sigma$,
$\iota\sim\sigma$ and $\sigma\sim\Dp$, which by the same identities say that $\Dm$,
$\Dp$ and $\Dc$ are bijections. As $\sigma$ is one by hypothesis, $\sigma\in\CB(G)$.  The
same three identities show that the remaining three edges follow from the bijectivity of
$\Dp$ and $\Dm$.

For the intersection, note first that $\sigma\notin\{\mathrm{id},\iota,\Dp,\Dm\}$: the
equalities $\sigma=\Dp$ and $\sigma=\Dm$ force $r=e$ for all $r$, while $\sigma=\mathrm{id}$
makes $\Dm$ and $\sigma=\iota$ makes $\Dp$ constant.  A further common vertex could only
arise from $\mathrm{id}=\iota$ or $\Dp=\Dm$, both of which force $\exp(G)\le2$, or from
$\mathrm{id}=\Dm$, that is $\sigma(r)=r^{2}$, or from $\Dp=\iota$, that is
$\sigma(r)=r^{-2}$.  In the last two cases $\Dp(r)=r^{3}$ and $\Dm(r)=r^{-3}$
respectively, which are not injective when $3$ divides $|G|$.
\end{proof}

\begin{rem}\label{rem:scm-not-cb}
If $\sigma$ is a strong complete mapping that is not a colouring bijection, the two spokes
$\sigma\sim\Dp$ and $\sigma\sim\Dm$ are absent. Both are governed by $\Dc$. Neither
triangle of the bowtie survives, and what is left is the path
$\Dp-\mathrm{id}-\sigma-\iota-\Dm$. This is what the third condition in
Definition~\ref{def:scm_bc} buys. Figure~\ref{fig:orth-even} shows the two pictures side
by side.

For $G=M_{16}$ the bowtie is exactly the local picture. An exhaustive check of all $4096$
colouring bijections of $M_{16}$ \cite{code} shows that none of the four remaining pairs
$\mathrm{id}\sim\iota$, $\mathrm{id}\sim\Dm$, $\iota\sim\Dp$, $\Dp\sim\Dm$ is ever an
edge. In general the bowtie is only the guaranteed part. In $D_8\times C_2$ some colouring
bijections do satisfy $\mathrm{id}\sim\Dm$ and $\iota\sim\Dp$. The two conditions coincide
there because that group has exponent $4$, so that $r^{-2}=r^{2}$.
\end{rem}
\medskip

The first equivalence of Corollary~\ref{cor:cb-witness} recovers a known observation.  For
$|G|$ odd the inversion map $\iota$ is the power orthomorphism $\varphi_{-1}$, and an
orthomorphism of $G$ is a strong complete mapping precisely when it is orthogonal to
$\varphi_{-1}$ \cite[\S2]{E4}.

\medskip

For groups of even order the global picture degenerates as well.  By
Lemma~\ref{lem:perm-cayley} the graph $\Perm(G)$ is
regular of degree $|\Orth(G)|$, so it is edgeless exactly when $G$ admits no complete
mapping, which by Hall--Paige happens precisely when the Sylow $2$-subgroup of $G$ is
nontrivial and cyclic. For $C_6$ and $S_3$ we get $\Orth(G)=\varnothing$, and no Latin
square based on the group is orthogonal to another, so that $\CB(G)=\varnothing$ as well.  For $G=C_2\times C_2$ the graph
$\Perm(G)$ has $24$ vertices and is $8$-regular, and $\Orth(G)$ is the complete bipartite
graph on $4+4$ vertices. Since it is triangle-free, $\omega(\Orth(G))=2$ and hence
$\omega(\Perm(G))=3$, in accordance with the maximum number $N(4)=3$ of mutually
orthogonal Latin squares of order $4$.  For $G=D_8$ and for $G=Q_8$ the graph $\Perm(G)$ is $384$-regular, but
$\Orth(G)$ has no edges at all, so $\omega(\Perm(G))=2$ and no three Latin squares based
on these groups are mutually orthogonal. Here $384$ counts orthomorphisms in the
unnormalised sense used throughout, that is, $48$ normalised ones.  These graphs, for all
five groups of order $8$, are determined in \cite{E6}.  Finally, in a group of exponent
$2$ one has $\Dp=\Dm$, so even the bowtie of Proposition~\ref{prop:bowtie} collapses.
This is one of the places where the order of $G$ matters.

\begin{figure}[htbp]
\centering
\begin{tikzpicture}[scale=0.95,
  vtx/.style ={circle,draw,thick,fill=black!6,inner sep=0pt,minimum size=8mm,font=\small},
  anch/.style={rectangle,draw,thick,fill=black!12,inner sep=1pt,minimum size=8mm,font=\small},
  ed/.style  ={thick},
  gone/.style={thick,densely dotted,black!30},
  lb/.style  ={font=\scriptsize,inner sep=1pt,fill=white},
  pl/.style  ={font=\footnotesize\itshape}
]
\def\bow{
  \coordinate (Bdp) at (-3.1,0.95);
  \coordinate (Bid) at (-1.55,-0.35);
  \coordinate (Bsi) at (0,0.95);
  \coordinate (Bio) at (1.55,-0.35);
  \coordinate (Bdm) at (3.1,0.95);
}
% ---------- (a) SCM, not CB ----------
\begin{scope}
  \bow
  \draw[ed] (Bdp)--(Bid) node[lb,midway]{$\sigma$};
  \draw[ed] (Bid)--(Bsi) node[lb,midway]{$\Dm$};
  \draw[ed] (Bsi)--(Bio) node[lb,midway]{$\Dp$};
  \draw[ed] (Bio)--(Bdm) node[lb,midway]{$\sigma$};
  \draw[gone] (Bsi)--(Bdp); \draw[gone] (Bsi)--(Bdm);
  \node[vtx]  at (Bdp) {$\Dp$};
  \node[anch] at (Bid) {$\mathrm{id}$};
  \node[vtx]  at (Bsi) {$\sigma$};
  \node[vtx]  at (Bio) {$\iota$};
  \node[vtx]  at (Bdm) {$\Dm$};
  \node[pl,align=center] at (0,-1.5) {(a) $\sigma\in\SCM(G)\setminus\CB(G)$: a path};
\end{scope}
% ---------- (b) CB ----------
\begin{scope}[yshift=-3.5cm]
  \bow
  \draw[ed] (Bdp)--(Bid) node[lb,midway]{$\sigma$};
  \draw[ed] (Bid)--(Bsi) node[lb,midway]{$\Dm$};
  \draw[ed] (Bsi)--(Bio) node[lb,midway]{$\Dp$};
  \draw[ed] (Bio)--(Bdm) node[lb,midway]{$\sigma$};
  \draw[ed] (Bsi)--(Bdp) node[lb,midway]{$\Dc$};
  \draw[ed] (Bsi)--(Bdm) node[lb,midway]{$(\Dc)^{-1}$};
  \node[vtx]  at (Bdp) {$\Dp$};
  \node[anch] at (Bid) {$\mathrm{id}$};
  \node[vtx]  at (Bsi) {$\sigma$};
  \node[vtx]  at (Bio) {$\iota$};
  \node[vtx]  at (Bdm) {$\Dm$};
  \node[pl,align=center] at (0,-1.5) {(b) $\sigma\in\CB(G)$: a bowtie};
\end{scope}
\end{tikzpicture}

\caption{Groups of even order: what the third condition buys
(Proposition~\ref{prop:bowtie} and Remark~\ref{rem:scm-not-cb}), drawn for $G=M_{16}$.
Dotted lines are the two absent spokes.  By Proposition~\ref{prop:bowtie} the bowtie in
(b) characterises colouring bijections: a bijection $\sigma$ lies in $\CB(G)$ precisely
when this picture can be drawn.}
\label{fig:orth-even}
\end{figure}

\subsection{The five canonical vertices}

For groups of odd order $\mathrm{id}$ and $\iota$ are themselves adjacent, so the first
equivalence of Corollary~\ref{cor:cb-witness} takes the form
\[
\sigma\in\SCM(G)\iff\{\mathrm{id},\iota,\sigma\}\ \text{is a triangle of}\ \Perm(G),
\]
while by Proposition~\ref{prop:bowtie} a colouring bijection is such a $\sigma$ for which
the bowtie closes as well.  Both statements, and the two triangles of the bowtie, sit
inside a single five-vertex configuration.

\begin{prop}\label{prop:odd-triangles}
Let $G$ be a finite group of odd order, let $\sigma\in\CB(G)$ and let $\iota(x)=x^{-1}$.
Then $\mathrm{id},\iota,\Dm,\Dp,\sigma$ are five vertices of $\Perm(G)$ which span a
wheel: the first four form a cycle
\[
\mathrm{id}\ \sim\ \iota\ \sim\ \Dm\ \sim\ \Dp\ \sim\ \mathrm{id},
\]
whose edges are witnessed by
\[
\begin{aligned}
w(\mathrm{id},\iota)&=r^{-2}, &\qquad w(\iota,\Dm)&=\sigma,\\
w(\Dm,\Dp)&=(\Dc)^{2}, &\qquad w(\mathrm{id},\Dp)&=\sigma,
\end{aligned}
\]
while $\sigma$ is adjacent to each of the four, the spokes being witnessed by
\[
\begin{aligned}
w(\mathrm{id},\sigma)&=\Dm, &\qquad w(\iota,\sigma)&=\Dp,\\
w(\sigma,\Dm)&=(\Dc)^{-1}, &\qquad w(\sigma,\Dp)&=\Dc .
\end{aligned}
\]
The two diagonals of the cycle are the only pairs whose adjacency is not automatic, and
each of them joins a vertex $\vartheta$ to the witness $w(\vartheta,\sigma)$ of that
vertex's own edge with $\sigma$:
\[
\mathrm{id}\sim\Dm=w(\mathrm{id},\sigma)
\iff r\mapsto r^{-2}\sigma(r)\ \text{is a bijection},
\]
\[
\iota\sim\Dp=w(\iota,\sigma)
\iff r\mapsto r^{2}\sigma(r)\ \text{is a bijection}.
\]
Consequently the wheel itself, that is, the rim together with the spokes, contains exactly
four triangles, namely $\sigma$ together with a rim edge, and hence four triples of
mutually orthogonal Latin squares based on $G$. Two of these triples contain the Cayley
table and two do not.  In particular the wheel contains no $K_4$: a $K_4$ can arise only
when one of the diagonals is present.
\end{prop}

Here, and in the displays above, a witness is determined only up to pointwise inversion,
according to the order in which the two endpoints are taken.  This is immaterial, since a
map is a bijection precisely when its pointwise inverse is.

\begin{proof}
All five maps are bijections: $\iota$ always, and $\sigma,\Dp,\Dm$ because
$\sigma\in\CB(G)$.  Each edge is verified by computing the witness $w(\theta,\varphi)$, using
$\iota(r)^{-1}=r$:
\[
\begin{aligned}
\mathrm{id}\sim\iota&: r^{-1}r^{-1}=r^{-2}, &
\iota\sim\Dm&: r\,r^{-1}\sigma(r)=\sigma(r),\\
\Dm\sim\Dp&: \sigma(r)^{-1}r\,r\,\sigma(r)=\Dc(r)^{2}, &
\mathrm{id}\sim\Dp&: r^{-1}r\,\sigma(r)=\sigma(r),\\
\mathrm{id}\sim\sigma&: r^{-1}\sigma(r)=\Dm(r), &
\iota\sim\sigma&: r\,\sigma(r)=\Dp(r),\\
\sigma\sim\Dm&: \sigma(r)^{-1}r^{-1}\sigma(r)=\Dc(r)^{-1}, &
\sigma\sim\Dp&: \sigma(r)^{-1}r\,\sigma(r)=\Dc(r).
\end{aligned}
\]
Each of these is a bijection: $r\mapsto r^{-2}$ because $|G|$ is odd, and the others
because $\sigma,\Dp,\Dm,\Dc$ are bijections and inversion is one, so that the maps listed
are, up to inversion, exactly $\sigma$, $\Dm$, $\Dp$, $\Dc$, $(\Dc)^{-1}$ and
$(\Dc)^{2}$.  For the two diagonals the same computation gives
$\mathrm{id}(r)^{-1}\Dm(r)=r^{-2}\sigma(r)$ and $\iota(r)^{-1}\Dp(r)=r^{2}\sigma(r)$,
neither of which follows from $\sigma\in\CB(G)$.  Finally, a wheel with a four-cycle as
rim has exactly four triangles, each consisting of the hub and a rim edge. Two of the rim
edges contain $\mathrm{id}$, and the assertion about Latin squares is
Lemma~\ref{lem:cliques-mols}(1).
\end{proof}

\begin{cor}\label{cor:exp3}
Let $G$ be a group of exponent $3$ with $|G|>3$, and let $\sigma\in\CB(G)$.  Then both
diagonals of the wheel of Proposition~\ref{prop:odd-triangles} are present, so that the
five vertices $\mathrm{id},\iota,\sigma,\Dp,\Dm$ are pairwise distinct and induce a
$K_5$ in $\Perm(G)$.  Consequently
\[
L_{\mathrm{id}},\qquad L_\iota,\qquad L_\sigma,\qquad L_{\Dp},\qquad L_{\Dm}
\]
are five mutually orthogonal Latin squares of order $|G|$ based on $G$.  Moreover
$\CB(G)\neq\varnothing$ for every such $G$.
\end{cor}

\begin{proof}
In a group of exponent $3$ one has $r^{-2}=r$ and $r^{2}=r^{-1}$, so the two conditional
maps of Proposition~\ref{prop:odd-triangles} are
\[
r\mapsto r^{-2}\sigma(r)=r\,\sigma(r)=\Dp(r),
\qquad
r\mapsto r^{2}\sigma(r)=r^{-1}\sigma(r)=\Dm(r),
\]
and both are bijections because $\sigma\in\CB(G)$.  Hence all ten pairs are edges.

For the distinctness, note first that $\sigma\neq\mathrm{id}$ and $\sigma\neq\iota$, since
otherwise $\Dm$ or $\Dp$ would be the constant map $r\mapsto e$.  Next, $\Dp=\sigma$ or
$\Dm=\sigma$ would force $r=e$ for every $r$, while $\Dp=\mathrm{id}$ would make $\sigma$
constant.  Each of $\Dp=\iota$, $\Dm=\iota$ and $\Dm=\mathrm{id}$ gives $\sigma=\mathrm{id}$,
using $r^{-2}=r$, which has just been excluded.  Finally $\mathrm{id}=\iota$ and
$\Dp=\Dm$ both force $r^{2}=e$ for all $r$, contradicting $\exp(G)=3$ and $|G|>3$.
Distinct bijections give distinct translation squares, so
Lemma~\ref{lem:cliques-mols}(1) yields the five mutually orthogonal squares.

A group of exponent $3$ with $|G|>3$ is a noncyclic $3$-group, and it is not isomorphic
to any $M_{3^{r}}$, whose exponent is $3^{\,r-1}\ge9$. Hence $\CB(G)\neq\varnothing$ by
Theorem~\ref{thm:main}.
\end{proof}

Figure~\ref{fig:orth-local-odd} shows this configuration.  Each node is a Latin
translation square based on $G$ and each line an orthogonal pair, the square node being
the Cayley table. Since $\Orth(G)$ is the neighbourhood of $\mathrm{id}$, a node lies in
$\Orth(G)$ exactly when it is joined to the square node.  Panel (a) is the wheel itself,
with the four spokes labelled by their witnesses and the two conditional diagonals
dashed. Panel (b) is the situation of Corollary~\ref{cor:exp3} and panel (c) that of
Remark~\ref{rem:Dm-not-vertex} below.  A wheel has no $K_4$, which is why
Corollary~\ref{cor:mols-3groups} yields three mutually orthogonal squares and not four.

\begin{figure}[htbp]
\centering
\begin{tikzpicture}[scale=0.95,
  vtx/.style ={circle,draw,thick,fill=black!6,inner sep=0pt,minimum size=8mm,font=\small},
  anch/.style={rectangle,draw,thick,fill=black!12,inner sep=1pt,minimum size=8mm,font=\small},
  ed/.style  ={thick},
  cond/.style={thick,densely dashed,black!55},
  lb/.style  ={font=\scriptsize,inner sep=1pt,fill=white},
  pl/.style  ={font=\footnotesize\itshape}
]
\def\rim{
  \coordinate (Aid) at (-1.65,1.45);
  \coordinate (Aio) at (1.65,1.45);
  \coordinate (Adm) at (1.65,-1.45);
  \coordinate (Adp) at (-1.65,-1.45);
  \coordinate (Asi) at (0,0);
}
% ---------- (a) ----------
\begin{scope}
  \rim
  \draw[ed] (Aid)--(Aio) node[lb,midway]{$r^{-2}$};
  \draw[ed] (Aio)--(Adm) node[lb,midway]{$\sigma$};
  \draw[ed] (Adm)--(Adp) node[lb,midway]{$(\Dc)^{2}$};
  \draw[ed] (Adp)--(Aid) node[lb,midway]{$\sigma$};
  \draw[ed] (Asi)--(Aid) node[lb,pos=.55]{$\Dm$};
  \draw[ed] (Asi)--(Aio) node[lb,pos=.55]{$\Dp$};
  \draw[ed] (Asi)--(Adm) node[lb,pos=.55]{$(\Dc)^{-1}$};
  \draw[ed] (Asi)--(Adp) node[lb,pos=.55]{$\Dc$};
  \draw[cond] (Aid) to[bend left=55] (Adm);
  \draw[cond] (Aio) to[bend left=55] (Adp);
  \node[anch] at (Aid) {$\mathrm{id}$};
  \node[vtx]  at (Aio) {$\iota$};
  \node[vtx]  at (Adm) {$\Dm$};
  \node[vtx]  at (Adp) {$\Dp$};
  \node[vtx]  at (Asi) {$\sigma$};
  \node[pl,align=center] at (0,-2.5) {(a) every $\sigma\in\CB(G)$};
\end{scope}
% ---------- (b) ----------
\begin{scope}[xshift=5.3cm]
  \rim
  \draw[ed] (Aid)--(Aio); \draw[ed] (Aio)--(Adm); \draw[ed] (Adm)--(Adp); \draw[ed] (Adp)--(Aid);
  \draw[ed] (Asi)--(Aid); \draw[ed] (Asi)--(Aio); \draw[ed] (Asi)--(Adm); \draw[ed] (Asi)--(Adp);
  \draw[ed] (Aid) to[bend left=55] (Adm);
  \draw[ed] (Aio) to[bend left=55] (Adp);
  \node[anch] at (Aid) {$\mathrm{id}$};
  \node[vtx]  at (Aio) {$\iota$};
  \node[vtx]  at (Adm) {$\Dm$};
  \node[vtx]  at (Adp) {$\Dp$};
  \node[vtx]  at (Asi) {$\sigma$};
  \node[pl,align=center] at (0,-2.5) {(b) $\exp(G)=3$:\\ $K_5$, hence $5$ MOLS};
\end{scope}
% ---------- (c) ----------
\begin{scope}[xshift=10.6cm]
  \rim
  \draw[ed] (Aid)--(Aio); \draw[ed] (Aio)--(Adm); \draw[ed] (Adm)--(Adp); \draw[ed] (Adp)--(Aid);
  \draw[ed] (Asi)--(Aid); \draw[ed] (Asi)--(Aio); \draw[ed] (Asi)--(Adm); \draw[ed] (Asi)--(Adp);
  \node[anch] at (Aid) {$\mathrm{id}$};
  \node[vtx]  at (Aio) {$\iota$};
  \node[vtx]  at (Adm) {$\Dm$};
  \node[vtx]  at (Adp) {$\Dp$};
  \node[vtx]  at (Asi) {$\sigma$};
  \node[pl,align=center] at (0,-2.5) {(c) $G=M_{27}$:\\ pure wheel};
\end{scope}
\end{tikzpicture}

\caption{The wheel spanned by a colouring bijection, $|G|$ odd
(Proposition~\ref{prop:odd-triangles}); dashed lines are the two conditional diagonals.}
\label{fig:orth-local-odd}
\end{figure}

\begin{rem}\label{rem:Dm-not-vertex}
Of the four triangles of the wheel, the two containing $\mathrm{id}$ correspond to edges
of $\Orth(G)$ together with the anchor. The other two need not lie in $\Orth(G)$ at all,
since $\Dm$ is a vertex of $\Orth(G)$ exactly when the diagonal
$\mathrm{id}\sim\Dm$ is present.  Both conditional edges are present whenever $\exp(G)=3$, by
Corollary~\ref{cor:exp3}, and then $\{\mathrm{id},\iota,\sigma,\Dp,\Dm\}$ induces a $K_5$
in $\Perm(G)$.  Both can also fail: for the colouring bijection $\sigma$ of $M_{27}$ displayed in
Table~\ref{tab:H3-L3-add-fixing-id}, and for every one of the $1458$ bijections in its
$\Hol$-orbit, the map $r\mapsto r^{-2}\sigma(r)$ is not injective \cite{code}, so there
$\Dm\notin\Orth(M_{27})$ although $\{\sigma,\Dp,\Dm\}$ is still a triangle of
$\Perm(M_{27})$.  See Figure~\ref{fig:orth-local-odd}.
\end{rem}

\FloatBarrier

\section{Colouring bijections as triple transversals}\label{sec:quandle}

We give one more interpretation of the definition of a colouring bijection. It rests on
three observations. First, $\Dp$, $\Dm$ and $\Dc$ are the diagonal maps of three binary
operations on $G$. Second, the corresponding condition is a simultaneous transversal
condition in three arrays. These arrays are the multiplication table, the division table,
and the operation table of the \emph{conjugation quandle} of $G$. Third, the last of these
arrays is not a Latin square. This places the notion outside the classical Latin-square
framework in which complete mappings live.

\subsection{Three operations}

Besides its multiplication, every finite group is equipped with two further binary
operations that will concern us:
\[
x\cdot y,
\qquad
x\backslash y:=x^{-1}y,
\qquad
x\triangleright y:=y^{-1}xy .
\]
The second is the \emph{left division} of the quasigroup $(G,\cdot)$. The
third is conjugation, written as a binary operation.  For a bijection
$\sigma\colon G\to G$ the three maps of Definition~\ref{def:scm_bc} are
precisely the diagonals of these operations along $\sigma$:
\begin{equation}\label{eq:three-diagonals}
\Dp_\sigma(x)=x\cdot\sigma(x),
\qquad
\Dm_\sigma(x)=x\backslash\sigma(x),
\qquad
\Dc_\sigma(x)=x\triangleright\sigma(x).
\end{equation}
Thus a complete mapping is a bijection whose diagonal along the multiplication is again a
bijection. A strong complete mapping asks the same for the multiplication and for the
division. A colouring bijection, in turn, asks it for all three operations.

\subsection{The conjugation quandle}

The operation $\triangleright$ is neither associative nor unital, so
$(G,\triangleright)$ is not a group.  It is, however, an instance of a
classical algebraic structure.

\begin{defn}\label{def:quandle}
A \emph{quandle} is a set $X$ together with a binary operation
$\triangleright$ satisfying, for all $x,y,z\in X$:
\begin{enumerate}
\item[\textup{(Q1)}] $x\triangleright x=x$ \hfill (idempotency);
\item[\textup{(Q2)}] the right translation
      $R_y\colon x\mapsto x\triangleright y$ is a bijection of $X$
      \hfill (right invertibility);
\item[\textup{(Q3)}] $(x\triangleright y)\triangleright z
      =(x\triangleright z)\triangleright(y\triangleright z)$
      \hfill (right self-distributivity).
\end{enumerate}
\end{defn}

Quandles were introduced independently by Joyce \cite{J} and Matveev
\cite{Mt} as algebraic invariants of knots. Axioms (Q1)--(Q3) are the
algebraic shadow of the three Reidemeister moves.  Every group provides an
example, and it is the one we need.

\begin{ex}\label{ex:conj-quandle}
For a group $G$ the pair $\Conj(G):=(G,\triangleright)$ with
$x\triangleright y=y^{-1}xy$ is a quandle, the \emph{conjugation quandle} of $G$.
Axiom (Q1) reads $x^{-1}xx=x$.  For (Q2), the right translation
$R_y=\rho_{y^{-1}}$ is an inner automorphism of $G$, hence a bijection.  For (Q3),
both sides equal $z^{-1}y^{-1}xyz$: on the left
$(x\triangleright y)\triangleright z=z^{-1}(y^{-1}xy)z$, while on the right
\[
(x\triangleright z)\triangleright(y\triangleright z)
=(z^{-1}yz)^{-1}(z^{-1}xz)(z^{-1}yz)
=z^{-1}y^{-1}z\cdot z^{-1}xz\cdot z^{-1}yz
=z^{-1}y^{-1}xyz .
\]
\end{ex}

Two features of $\Conj(G)$ matter here.

First, axiom (Q2) states that every right translation $R_y=\rho_{y^{-1}}$ is a
permutation of $G$.  This is precisely what turns the third defining condition into a
transversal condition, as we shall see in Subsection~\ref{subsec:triple}.

Second, the orbits of the group generated by the right translations
$R_y$ --- the inner automorphism group of $G$ --- are the conjugacy
classes.  The quandle $\Conj(G)$ therefore decomposes into orbits indexed
by the conjugacy classes of $G$, and this decomposition governs the
third condition (Proposition~\ref{prop:class-local}).  If $G$ is abelian,
then $x\triangleright y=x$, so $\Conj(G)$ is the \emph{trivial} quandle, all
its right translations are the identity, and the third condition carries no
information. This is Remark~\ref{rem:abelian} seen from the quandle side.

\subsection{The triple transversal theorem}\label{subsec:triple}

The three conditions of Definition~\ref{def:scm_bc} become a single condition on a
single set of cells, once the three operations of \eqref{eq:three-diagonals} are laid
out as arrays.  Consider the three
arrays on $G\times G$
\begin{equation}\label{eq:three-tables}
L(r,c)=rc,
\qquad
L^{-}(r,c)=r^{-1}c,
\qquad
Q(r,c)=c^{-1}rc=r\triangleright c ,
\end{equation}
the multiplication table, the division table and the operation table of
$\Conj(G)$.  Here $L(r,c)=rc$ is the transpose of the translation square
$L_{\mathrm{id}}(r,c)=cr$ of Section~\ref{sec:mols}. The transpose is used so that the
diagonal $L(x,\sigma(x))=x\sigma(x)$ reproduces $\Dp_\sigma$ directly.  The notion of a
transversal recalled in the Introduction makes sense verbatim for an arbitrary array on
$G\times G$. We call such an array \emph{column-Latin} if each of its columns is a
permutation of $G$.  The tables $L$ and $L^{-}$ are Latin
squares (they are two of the six parastrophes of the quasigroup
$(G,\cdot)$).  The table $Q$ is \emph{not} a Latin square: its column $c$
is the inner automorphism $r\mapsto c^{-1}rc$, which is a permutation ---
this is precisely axiom (Q2) --- whereas its row $r$ takes as values the
conjugacy class of $r$, each element repeated $|C_G(r)|$ times.  Hence $Q$
is column-Latin, and it is Latin only for the trivial group.  Axiom (Q2)
is thus exactly the statement that transversals of $Q$ are a meaningful
notion at all.

\begin{thm}\label{thm:triple}
A bijection $\sigma\colon G\to G$ is a colouring bijection if and only if
the cell set
\[
T_\sigma=\{(x,\sigma(x))\ :\ x\in G\}
\]
is a common transversal of the three arrays $L$, $L^{-}$ and $Q$ of
\eqref{eq:three-tables}.
\end{thm}

\begin{proof}
The set $T_\sigma$ meets every row exactly once because $\sigma$ is a map,
and every column exactly once because $\sigma$ is a bijection.  By
\eqref{eq:three-diagonals} its symbol sequences in the three arrays are
\[
L(x,\sigma(x))=\Dp_\sigma(x),
\qquad
L^{-}(x,\sigma(x))=\Dm_\sigma(x),
\qquad
Q(x,\sigma(x))=\Dc_\sigma(x).
\]
These symbols are pairwise distinct in all three arrays precisely when
$\Dp_\sigma$, $\Dm_\sigma$ and $\Dc_\sigma$ are bijections.
\end{proof}

\medskip

A single colouring bijection does more than produce one transversal.

\begin{prop}\label{prop:mate}
Let $\sigma\in\CB(G)$ and put $N(x,y)=\sigma(x)^{-1}y$.  Then:
\begin{enumerate}
\item $N$ is a Latin square, orthogonal to $L$ and to $L^{-}$
      simultaneously;
\item the map $(x,y)\mapsto\bigl(Q(x,y),N(x,y)\bigr)$ is a bijection of
      $G\times G$;
\item every symbol class of $N$, that is, every translate
      $T_{\sigma g}=\{(x,\sigma(x)g):x\in G\}$ with $g\in G$, is again a
      common transversal of $(L,L^{-},Q)$.  Hence a single colouring
      bijection resolves $G\times G$ into $|G|$ pairwise disjoint triple
      transversals.
\end{enumerate}
\end{prop}

\begin{proof}
(1) $N$ is obtained from $L$ by the row relabelling
$x\mapsto\sigma(x)^{-1}$, hence is a Latin square.  Suppose
$N(x_1,y_1)=N(x_2,y_2)$. Then $y_2=\sigma(x_2)\sigma(x_1)^{-1}y_1$.
Substituting this into $L(x_1,y_1)=L(x_2,y_2)$ gives
$x_1\sigma(x_1)=x_2\sigma(x_2)$, so $x_1=x_2$ by injectivity of
$\Dp_\sigma$, and then $y_1=y_2$. Thus orthogonality of $N$ and $L$ follows.
The same substitution in $L^{-}(x_1,y_1)=L^{-}(x_2,y_2)$ gives
$x_1^{-1}\sigma(x_1)=x_2^{-1}\sigma(x_2)$, and injectivity of $\Dm_\sigma$
applies.

(2) With the same substitution, $Q(x_1,y_1)=Q(x_2,y_2)$ reduces to
$\sigma(x_1)^{-1}x_1\sigma(x_1)=\sigma(x_2)^{-1}x_2\sigma(x_2)$, that is,
to $\Dc_\sigma(x_1)=\Dc_\sigma(x_2)$, whence $x_1=x_2$ and $y_1=y_2$.  A
count of the $|G|^{2}$ cells finishes the argument.

(3) The translate $T_{\sigma g}$ is $T_{\sigma_g}$ for
$\sigma_g(x)=\sigma(x)g$, and $\sigma_g=\sigma^{(g,\mathrm{id})}$ is the
image of $\sigma$ under the right-translation part of the holomorph
action. By Remark~\ref{rem:hol-special}(2),
\[
\Dp_{\sigma_g}=\Rt{g}\circ\Dp_\sigma,
\qquad
\Dm_{\sigma_g}=\Rt{g}\circ\Dm_\sigma,
\qquad
\Dc_{\sigma_g}=\conjg{g}\circ\Dc_\sigma,
\]
all of which are bijections.  Thus $\sigma_g\in\CB(G)$, and
Theorem~\ref{thm:triple} applies.  The cell sets $T_{\sigma g}$, $g\in G$,
are pairwise disjoint and cover $G\times G$.
\end{proof}

\subsection{Locality of the third condition}

The condition on $\Dc_\sigma$ is concentrated exactly where $G$ fails to be
commutative.

\begin{prop}\label{prop:class-local}
For every bijection $\sigma$ and every $x\in G$, the element
$\Dc_\sigma(x)=x\triangleright\sigma(x)$ is conjugate to $x$.
Consequently $\Dc_\sigma$ is a bijection if and only if for every
conjugacy class $K$ of $G$ the restriction $\Dc_\sigma|_K\colon K\to K$ is
a bijection.  Equivalently, for every class $K$ and every $k\in K$ there is
exactly one $r\in K$ with $\sigma(r)\in T_{r\to k}$, where
\[
T_{r\to k}=\{c\in G:\ c^{-1}rc=k\}
\]
is empty or a right coset of the centraliser $C_G(r)$.
\end{prop}

\begin{proof}
The first assertion is immediate from
$\Dc_\sigma(x)=\sigma(x)^{-1}x\,\sigma(x)$.  A class-preserving map of a
finite group is therefore a bijection of $G$ if and only if it is a
bijection on each class.  Finally, $\sigma(r)\in T_{r\to k}$ says exactly
that $\Dc_\sigma(r)=k$, and the transporter $T_{r\to k}$ is empty or a
right coset of $C_G(r)$.
\end{proof}

\section{The holomorph and its action}\label{sec:hol}

Colouring bijections are never isolated: they come in orbits under a
large group of symmetries.  This section makes that precise.  The action
of $\Aut(G)$ by conjugation is the special case $g=e$ below.

For $g\in G$ let
\[
\Rt{g}\colon G\to G,\quad \Rt{g}(u)=ug,
\qquad\text{and}\qquad
\conjg{g}\colon G\to G,\quad \conjg{g}(u)=g^{-1}ug,
\]
denote right translation by $g$ and conjugation by $g$. Both are
bijections of $G$, and $\conjg{g}=\rho_{g^{-1}}$ is an automorphism.

\begin{defn}\label{def:hol}
The \emph{holomorph} of $G$ is the semidirect product
\[
\Hol(G)=G\rtimes\Aut(G),
\qquad
(g,f)\cdot(h,\varphi)=\bigl(g\cdot f(h),\ f\circ\varphi\bigr).
\]
For $\sigma\in\Sym(G)$ and $(g,f)\in\Hol(G)$ put
\begin{equation}\label{eq:hol-action}
\sigma^{(g,f)}
:=f^{-1}\circ \Rt{g}\circ\sigma\circ f,
\qquad\text{i.e.}\qquad
\sigma^{(g,f)}(x)=f^{-1}\bigl(\sigma(f(x))\cdot g\bigr).
\end{equation}
Two special cases are:
\begin{itemize}
\item \emph{automorphism conjugation} ($g=e$):
      $\sigma^{(e,f)}=f^{-1}\circ\sigma\circ f$;
\item \emph{right translation} ($f=\mathrm{id}$):
      $\sigma^{(g,\mathrm{id})}(x)=\sigma(x)\cdot g$.
\end{itemize}
\end{defn}

\begin{lem}\label{lem:hol-is-action}
Formula \eqref{eq:hol-action} defines a right action of $\Hol(G)$ on
$\Sym(G)$.
\end{lem}

\begin{proof}
Clearly $\sigma^{(e,\mathrm{id})}=\sigma$.  Since $f$ is an automorphism,
$f\circ \Rt{h}\circ f^{-1}=\Rt{f(h)}$, and right translations satisfy
$\Rt{g}\circ \Rt{h}=\Rt{hg}$.  Hence, for $(g,f),(h,\varphi)\in\Hol(G)$,
\[
\begin{aligned}
\bigl(\sigma^{(g,f)}\bigr)^{(h,\varphi)}
&=\varphi^{-1}\Rt{h}\bigl(f^{-1}\Rt{g}\,\sigma f\bigr)\varphi
 =(f\varphi)^{-1}\bigl(f\,\Rt{h}\,f^{-1}\bigr)\Rt{g}\,\sigma\,(f\varphi)\\
&=(f\varphi)^{-1}\,\Rt{f(h)}\Rt{g}\,\sigma\,(f\varphi)
 =(f\varphi)^{-1}\,\Rt{g\cdot f(h)}\,\sigma\,(f\varphi),
\end{aligned}
\]
which is exactly $\sigma^{(g,f)\cdot(h,\varphi)}$.
\end{proof}

\begin{lem}\label{lem:hol_scm}
If $\sigma\in\SCM(G)$ and $(g,f)\in\Hol(G)$, then
$\sigma^{(g,f)}\in\SCM(G)$.
\end{lem}

\begin{proof}
Put $\tau=\sigma^{(g,f)}$. As a composition of bijections, $\tau$ is a
bijection.  Substitute $y=f(x)$ and use that $f^{-1}$ is a homomorphism:
\[
\Dp_\tau(x)=x\cdot\tau(x)
=f^{-1}(y)\cdot f^{-1}\bigl(\sigma(y)g\bigr)
=f^{-1}\bigl(y\,\sigma(y)\,g\bigr)
=f^{-1}\Bigl(\Rt{g}\bigl(\Dp_\sigma(y)\bigr)\Bigr),
\]
\[
\Dm_\tau(x)=x^{-1}\cdot\tau(x)
=f^{-1}(y^{-1})\cdot f^{-1}\bigl(\sigma(y)g\bigr)
=f^{-1}\bigl(y^{-1}\sigma(y)\,g\bigr)
=f^{-1}\Bigl(\Rt{g}\bigl(\Dm_\sigma(y)\bigr)\Bigr).
\]
Thus
$\Dp_\tau=f^{-1}\circ \Rt{g}\circ\Dp_\sigma\circ f$ and
$\Dm_\tau=f^{-1}\circ \Rt{g}\circ\Dm_\sigma\circ f$
are compositions of bijections.
\end{proof}

\begin{lem}[Holomorph action on colouring bijections]\label{lem:hol_bc}
If $\sigma\in\CB(G)$ and $(g,f)\in\Hol(G)$, then
$\sigma^{(g,f)}\in\CB(G)$.  In particular $\Hol(G)$ acts on $\CB(G)$.
\end{lem}

\begin{proof}
By Lemma~\ref{lem:hol_scm}, $\tau=\sigma^{(g,f)}\in\SCM(G)$, so only
$\Dc_\tau$ has to be checked.  With $y=f(x)$,
\begin{align*}
\Dc_\tau(x)
&=\tau(x)^{-1}\cdot x\cdot\tau(x)
=f^{-1}\bigl(g^{-1}\sigma(y)^{-1}\bigr)\cdot f^{-1}(y)\cdot
 f^{-1}\bigl(\sigma(y)g\bigr)\\
&=f^{-1}\bigl(g^{-1}\cdot\sigma(y)^{-1}y\,\sigma(y)\cdot g\bigr)
=f^{-1}\Bigl(\conjg{g}\bigl(\Dc_\sigma(y)\bigr)\Bigr).
\end{align*}
Hence $\Dc_\tau=f^{-1}\circ\conjg{g}\circ\Dc_\sigma\circ f$ is a
composition of bijections.
\end{proof}

\begin{cor}\label{cor:orbits}
Both $\SCM(G)$ and $\CB(G)$ are unions of $\Hol(G)$-orbits, and for every
$\sigma\in\CB(G)$
\[
|\Orb(\sigma)|=\frac{|\Hol(G)|}{|\Stab_{\Hol(G)}(\sigma)|},
\qquad
|\Hol(G)|=|G|\cdot|\Aut(G)| .
\]
In particular, a single colouring bijection yields an orbit of
$|\Hol(G)|/|\Stab(\sigma)|$ colouring bijections, hence
$|\Hol(G)|/|\Stab(\sigma)|-1$ further ones, with no additional
verification.
\end{cor}

\begin{rem}\label{rem:hol-special}
For the two special cases of Definition~\ref{def:hol} the formulas of
Lemmas~\ref{lem:hol_scm} and~\ref{lem:hol_bc} read:
\begin{enumerate}
\item $g=e$:\quad
      $\Delta^{\,\bullet}_\tau=f^{-1}\circ\Delta^{\,\bullet}_\sigma\circ f$
      for $\bullet\in\{+,-,c\}$;
\item $f=\mathrm{id}$:\quad
      $\Dp_\tau=\Rt{g}\circ\Dp_\sigma$, \
      $\Dm_\tau=\Rt{g}\circ\Dm_\sigma$, \
      $\Dc_\tau=\conjg{g}\circ\Dc_\sigma$.
\end{enumerate}
Case (1) is the conjugation action of $\Aut(G)$ on $\CB(G)$.
\end{rem}

\medskip

The computations above say more than that $\Hol(G)$ permutes $\CB(G)$: the holomorph acts
on the whole translation graph of Definition~\ref{def:perm-graph}.

\begin{lem}[Holomorph action on $\Perm(G)$]\label{lem:hol-orth}
Let $G$ be a finite group and let $(g,f)\in\Hol(G)$.  Then for all bijections
$\theta,\varphi$ of $G$,
\[
\theta\sim\varphi
\qquad\Longleftrightarrow\qquad
\theta^{(g,f)}\sim\varphi^{(g,f)} .
\]
Hence the holomorph acts on $\Perm(G)$ by graph automorphisms.  Since
\eqref{eq:hol-action} is a right action, the assignment
$\alpha\colon(g,f)\mapsto\bigl(\sigma\mapsto\sigma^{(g,f)}\bigr)$ satisfies
$\alpha_{uv}=\alpha_{v}\circ\alpha_{u}$. Composing it with inversion in $\Hol(G)$
turns it into a homomorphism
\[
\Hol(G)\longrightarrow\Aut\bigl(\Perm(G)\bigr),
\qquad
(g,f)\longmapsto\alpha_{(g,f)^{-1}} .
\]
\end{lem}

\begin{proof}
Put $\tau=\theta^{(g,f)}$, $\tau'=\varphi^{(g,f)}$ and substitute $y=f(x)$.  Both are
bijections, and since $f^{-1}$ is a homomorphism,
\[
\tau(x)^{-1}\tau'(x)
=f^{-1}\bigl(g^{-1}\theta(y)^{-1}\bigr)\cdot f^{-1}\bigl(\varphi(y)g\bigr)
=f^{-1}\Bigl(\conjg{g}\bigl(\theta(y)^{-1}\varphi(y)\bigr)\Bigr),
\]
that is,
\[
\bigl(r\mapsto\tau(r)^{-1}\tau'(r)\bigr)
=f^{-1}\circ\conjg{g}\circ
 \bigl(r\mapsto\theta(r)^{-1}\varphi(r)\bigr)\circ f .
\]
As $f^{-1}$ and $\conjg{g}$ are bijections of $G$, one side is a bijection if and only if
the other is.  By Lemma~\ref{lem:hol-is-action} each $(g,f)$ acts as a permutation of
$\Sym(G)$, which is the vertex set of $\Perm(G)$, and the displayed equivalence says that
this permutation preserves edges and non-edges.
\end{proof}

The action does not fix the Cayley table, but it moves it only within a class of
indistinguishable vertices.

\begin{lem}\label{lem:hol-anchor}
Let $G$ be a finite group.
\begin{enumerate}
\item $\mathrm{id}^{(g,f)}=\Rt{h}$ with $h=f^{-1}(g)$, so the $\Hol(G)$-orbit of
$\mathrm{id}$ is the set $\{\Rt{h}: h\in G\}$ of right translations, and the stabiliser
of $\mathrm{id}$ is $\Aut(G)\le\Hol(G)$.
\item The right translations are pairwise nonadjacent and have a common neighbourhood:
$N(\Rt{h})=N(\mathrm{id})=\Orth(G)$ for every $h\in G$.
\item Consequently $\Hol(G)$ preserves $\Orth(G)$ setwise, the induced action being the
restriction of \eqref{eq:hol-action}.
\end{enumerate}
\end{lem}

\begin{proof}
(1) $\mathrm{id}^{(g,f)}(x)=f^{-1}(f(x)g)=x\,f^{-1}(g)$, which is $\mathrm{id}$ if and
only if $g=e$. As $g$ runs over $G$ so does $f^{-1}(g)$.

(2) For $\theta$ a bijection, $r\mapsto\theta(r)^{-1}\Rt{h}(r)=\theta(r)^{-1}rh$ is a
bijection if and only if $r\mapsto\theta(r)^{-1}r$ is one, since right translation by $h$
is a bijection. The latter says $\theta\sim\mathrm{id}$.  Taking $\theta=\Rt{k}$ gives
$r\mapsto (rk)^{-1}(rh)=k^{-1}h$, a constant map, so distinct right translations are nonadjacent
(for $|G|>1$).

(3) Immediate from (1), (2) and Lemma~\ref{lem:hol-orth}: an automorphism of $\Perm(G)$
carrying $\mathrm{id}$ to $\Rt{h}$ carries $N(\mathrm{id})$ onto $N(\Rt{h})=N(\mathrm{id})$.
Alternatively, and independently of (1) and (2): a bijection $\theta$ is an orthomorphism
precisely when $\Dm_\theta$ is a bijection, and the second computation in the proof of
Lemma~\ref{lem:hol_scm}, which uses nothing beyond the bijectivity of $\theta$, gives
$\Dm_{\theta^{(g,f)}}=f^{-1}\circ \Rt{g}\circ\Dm_\theta\circ f$.
\end{proof}

\begin{prop}\label{prop:hol-orth-faithful}
For every finite group $G$ the action of Lemma~\ref{lem:hol-orth} is faithful.
Consequently
\[
\Hol(G)\hookrightarrow\Aut\bigl(\Perm(G)\bigr),
\qquad\text{so}\qquad
\bigl|\Aut\bigl(\Perm(G)\bigr)\bigr|\ \ge\ |G|\cdot|\Aut(G)| .
\]
\end{prop}

\begin{proof}
For $a\in G$ put $\theta_a(x)=x^{-1}a$. Each $\theta_a$ is a bijection, hence a vertex,
and no assumption on $|G|$ is needed for this.  Suppose that $(g,f)$ acts trivially, so
that $\theta_a^{(g,f)}=\theta_a$ for every $a\in G$.  Writing $y=f(x)$ we have
\[
\theta_a^{(g,f)}(x)=f^{-1}\bigl(y^{-1}a\,g\bigr),
\qquad
\theta_a(x)=x^{-1}a=f^{-1}\bigl(y^{-1}f(a)\bigr),
\]
so $\theta_a^{(g,f)}=\theta_a$ forces $a\,g=f(a)$ for every $a\in G$.  Taking $a=e$
yields $g=e$, and then $f=\mathrm{id}$.
\end{proof}

\begin{rem}\label{rem:hol-triangles}
By the displayed formulas in the proofs of Lemmas~\ref{lem:hol_scm}
and~\ref{lem:hol_bc} (see also Remark~\ref{rem:hol-special}), for every
$\sigma\in\CB(G)$ and $(g,f)\in\Hol(G)$,
\[
\Dp_{\sigma^{(g,f)}}=(\Dp_\sigma)^{(g,f)},
\qquad
\Dm_{\sigma^{(g,f)}}=(\Dm_\sigma)^{(g,f)} .
\]
The corresponding formula for $\Dc$ fails: by Lemma~\ref{lem:hol_bc} the translation
$\Rt{g}$ is there replaced by $\conjg{g}$, and since $\Dc_\sigma$ is onto,
$\Dc_{\sigma^{(g,f)}}=(\Dc_\sigma)^{(g,f)}$ would force $\Rt{g}=\conjg{g}$, that is
$g=e$.
Hence $\Hol(G)$ carries the triangle $\{\mathrm{id},\sigma,\Dp_\sigma\}$ of
Proposition~\ref{prop:mols} to the triangle
$\{\Rt{h},\sigma^{(g,f)},\Dp_{\sigma^{(g,f)}}\}$, $h=f^{-1}(g)$, and, for $|G|$ odd, the
triangle $\{\sigma,\Dp_\sigma,\Dm_\sigma\}$ of
Proposition~\ref{prop:odd-triangles} to the corresponding triangle of
$\sigma^{(g,f)}$: the whole five-vertex configuration of
Figure~\ref{fig:orth-local-odd} is $\Hol(G)$-equivariant, up to replacing $\mathrm{id}$
by $\Rt{h}$ and $\iota$ by $\Rt{h}\circ\iota$.  Indeed
$\iota^{(g,f)}(x)=f^{-1}\bigl(f(x)^{-1}g\bigr)=x^{-1}h$, so that, exactly as for
$\mathrm{id}$, the vertex $\iota$ is fixed precisely by the elements of $\Aut(G)$.  Thus one
colouring bijection produces an orbit of such configurations, of the size computed in
Corollary~\ref{cor:orbits}.
\end{rem}

\bigskip

\section{Colourability of groups of small order}\label{sec:order27}

By Theorem~\ref{3-group-abelian}, every noncyclic abelian $3$-group of
order at most $27$ is colourable, while by Remark~\ref{rem:cyclic-excluded}
no nontrivial cyclic $3$-group is.  Hence among the $3$-groups of order at
most $27$ the only ones whose colourability is not yet settled are the two
nonabelian groups of order $27$: the extraspecial group of exponent $3$,
\[
H_3=\langle x,y,z \mid x^{3}=y^{3}=z^{3}=e,\ [x,y]=z,\ [x,z]=[y,z]=e\rangle,
\]
and the modular group
\[
M_{27}=\langle a,b \mid a^{9}=b^{3}=e,\ bab^{-1}=a^{4}\rangle .
\]

\begin{lem}\label{lem:order27}
Both $H_3$ and $M_{27}$ are colourable.
\end{lem}

\begin{proof}
Explicit colouring bijections were found by computer search and are displayed in
Table~\ref{tab:H3-L3-add-fixing-id}; the code is available at \cite{code}.  Elements of
$H_3$ are written as triples $(a,b,c)\leftrightarrow x^{a}y^{b}z^{c}$,
with multiplication
\[
(a,b,c)\cdot(a',b',c')=(a+a',\,b+b',\,c+c'+ab')\pmod 3 ,
\]
and elements of $M_{27}$ as pairs $(u,j)\leftrightarrow a^{u}b^{j}$, with
\[
(u,j)\cdot(u',j')=(u+4^{\,j}u' \bmod 9,\ j+j' \bmod 3).
\]
For each group the table lists $\sigma(x)$ together with
$\Dp(x)=x\sigma(x)$, $\Dm(x)=x^{-1}\sigma(x)$ and
$\Dc(x)=\sigma(x)^{-1}x\sigma(x)$. Direct inspection shows that all eight
columns are permutations, so $\sigma\in\CB(H_3)$ and
$\sigma\in\CB(M_{27})$ respectively.
\end{proof}

\begin{longtable}{ccccc|ccccc}
\caption{Colouring bijections of $H_3$ (left, elements written as triples
$(a,b,c)\leftrightarrow x^{a}y^{b}z^{c}$) and of $M_{27}$ (right, elements written
as pairs $(u,j)\leftrightarrow a^{u}b^{j}$), together with
$\Dp(x)=x\sigma(x)$, $\Dm(x)=x^{-1}\sigma(x)$ and $\Dc(x)=\sigma(x)^{-1}x\sigma(x)$.
Each of the eight columns is a permutation of the group.}
\label{tab:H3-L3-add-fixing-id}\\
\toprule
$x$ & $\sigma(x)$ & $\Dp(x)$ & $\Dm(x)$ & $\Dc(x)$ &
$x$ & $\sigma(x)$ & $\Dp(x)$ & $\Dm(x)$ & $\Dc(x)$\\
\midrule
\endfirsthead
\multicolumn{10}{c}{\textit{Table \ref{tab:H3-L3-add-fixing-id} continued}}\\
\toprule
$x$ & $\sigma(x)$ & $\Dp(x)$ & $\Dm(x)$ & $\Dc(x)$ &
$x$ & $\sigma(x)$ & $\Dp(x)$ & $\Dm(x)$ & $\Dc(x)$\\
\midrule
\endhead
\bottomrule
\endfoot
$(0,0,0)$ & $(0,0,0)$ & $(0,0,0)$ & $(0,0,0)$ & $(0,0,0)$
&
$(0,0)$ & $(0,0)$ & $(0,0)$ & $(0,0)$ & $(0,0)$\\
$(0,0,1)$ & $(0,1,2)$ & $(0,1,0)$ & $(0,1,1)$ & $(0,0,1)$
&
$(0,1)$ & $(3,1)$ & $(3,2)$ & $(3,0)$ & $(0,1)$\\
$(0,0,2)$ & $(0,2,0)$ & $(0,2,2)$ & $(0,2,1)$ & $(0,0,2)$
&
$(0,2)$ & $(6,2)$ & $(6,1)$ & $(6,0)$ & $(0,2)$\\
$(0,1,0)$ & $(1,0,2)$ & $(1,1,2)$ & $(1,2,2)$ & $(0,1,2)$
&
$(1,0)$ & $(5,0)$ & $(6,0)$ & $(4,0)$ & $(1,0)$\\
$(0,1,1)$ & $(1,1,0)$ & $(1,2,1)$ & $(1,0,2)$ & $(0,1,0)$
&
$(1,1)$ & $(3,0)$ & $(4,1)$ & $(5,2)$ & $(1,1)$\\
$(0,1,2)$ & $(1,2,0)$ & $(1,0,2)$ & $(1,1,1)$ & $(0,1,1)$
&
$(1,2)$ & $(1,1)$ & $(8,0)$ & $(0,2)$ & $(4,2)$\\
$(0,2,0)$ & $(2,1,0)$ & $(2,0,0)$ & $(2,2,0)$ & $(0,2,2)$
&
$(2,0)$ & $(0,1)$ & $(2,1)$ & $(7,1)$ & $(5,0)$\\
$(0,2,1)$ & $(2,0,2)$ & $(2,2,0)$ & $(2,1,1)$ & $(0,2,0)$
&
$(2,1)$ & $(6,1)$ & $(8,2)$ & $(1,0)$ & $(5,1)$\\
$(0,2,2)$ & $(2,2,1)$ & $(2,1,0)$ & $(2,0,2)$ & $(0,2,1)$
&
$(2,2)$ & $(4,1)$ & $(3,0)$ & $(8,2)$ & $(2,2)$\\
$(1,0,0)$ & $(1,0,1)$ & $(2,0,1)$ & $(0,0,1)$ & $(1,0,0)$
&
$(3,0)$ & $(4,2)$ & $(7,2)$ & $(1,2)$ & $(3,0)$\\
$(1,0,1)$ & $(0,1,1)$ & $(1,1,0)$ & $(2,1,2)$ & $(1,0,2)$
&
$(3,1)$ & $(1,0)$ & $(7,1)$ & $(4,2)$ & $(6,1)$\\
$(1,0,2)$ & $(1,2,1)$ & $(2,2,2)$ & $(0,2,0)$ & $(1,0,1)$
&
$(3,2)$ & $(8,0)$ & $(5,2)$ & $(2,1)$ & $(6,2)$\\
$(1,1,0)$ & $(2,2,0)$ & $(0,0,2)$ & $(1,1,2)$ & $(1,1,0)$
&
$(4,0)$ & $(2,2)$ & $(6,2)$ & $(7,2)$ & $(7,0)$\\
$(1,1,1)$ & $(2,0,0)$ & $(0,1,1)$ & $(1,2,0)$ & $(1,1,2)$
&
$(4,1)$ & $(0,2)$ & $(4,0)$ & $(8,1)$ & $(7,1)$\\
$(1,1,2)$ & $(1,0,0)$ & $(2,1,2)$ & $(0,2,2)$ & $(1,1,1)$
&
$(4,2)$ & $(2,0)$ & $(0,2)$ & $(1,1)$ & $(7,2)$\\
$(1,2,0)$ & $(2,0,1)$ & $(0,2,1)$ & $(1,1,0)$ & $(1,2,2)$
&
$(5,0)$ & $(8,2)$ & $(4,2)$ & $(3,2)$ & $(2,0)$\\
$(1,2,1)$ & $(0,0,2)$ & $(1,2,0)$ & $(2,1,0)$ & $(1,2,1)$
&
$(5,1)$ & $(4,0)$ & $(3,1)$ & $(2,2)$ & $(8,1)$\\
$(1,2,2)$ & $(0,1,0)$ & $(1,0,0)$ & $(2,2,2)$ & $(1,2,0)$
&
$(5,2)$ & $(7,2)$ & $(0,1)$ & $(8,0)$ & $(8,2)$\\
$(2,0,0)$ & $(2,1,2)$ & $(1,1,1)$ & $(0,1,0)$ & $(2,0,2)$
&
$(6,0)$ & $(2,1)$ & $(8,1)$ & $(5,1)$ & $(6,0)$\\
$(2,0,1)$ & $(0,0,1)$ & $(2,0,2)$ & $(1,0,0)$ & $(2,0,1)$
&
$(6,1)$ & $(8,1)$ & $(2,2)$ & $(5,0)$ & $(3,1)$\\
$(2,0,2)$ & $(0,2,1)$ & $(2,2,1)$ & $(1,2,1)$ & $(2,0,0)$
&
$(6,2)$ & $(7,0)$ & $(1,2)$ & $(4,1)$ & $(3,2)$\\
$(2,1,0)$ & $(1,1,1)$ & $(0,2,0)$ & $(2,0,1)$ & $(2,1,1)$
&
$(7,0)$ & $(7,1)$ & $(5,1)$ & $(0,1)$ & $(4,0)$\\
$(2,1,1)$ & $(2,2,2)$ & $(1,0,1)$ & $(0,1,2)$ & $(2,1,0)$
&
$(7,1)$ & $(1,2)$ & $(2,0)$ & $(3,1)$ & $(4,1)$\\
$(2,1,2)$ & $(2,1,1)$ & $(1,2,2)$ & $(0,0,2)$ & $(2,1,2)$
&
$(7,2)$ & $(3,2)$ & $(1,1)$ & $(2,0)$ & $(1,2)$\\
$(2,2,0)$ & $(1,1,2)$ & $(0,0,1)$ & $(2,2,1)$ & $(2,2,0)$
&
$(8,0)$ & $(6,0)$ & $(5,0)$ & $(7,0)$ & $(8,0)$\\
$(2,2,1)$ & $(0,2,2)$ & $(2,1,1)$ & $(1,0,1)$ & $(2,2,2)$
&
$(8,1)$ & $(5,2)$ & $(1,0)$ & $(6,1)$ & $(2,1)$\\
$(2,2,2)$ & $(1,2,2)$ & $(0,1,2)$ & $(2,0,0)$ & $(2,2,1)$
&
$(8,2)$ & $(5,1)$ & $(7,0)$ & $(6,2)$ & $(5,2)$\\
\end{longtable}

\begin{rem}\label{rem:27-orbits}
By Corollary~\ref{cor:orbits} each of the two bijections of
Lemma~\ref{lem:order27} generates a whole $\Hol$-orbit of colouring
bijections.  We have
\[
\Aut(H_3)\cong (C_3\times C_3)\rtimes\GL(2,\F_3),\qquad |\Aut(H_3)|=432,
\]
so that $|\Hol(H_3)|=27\cdot432=11664$.
\[
\Aut(M_{27})\cong (C_3\times C_3)\rtimes C_6,\qquad |\Aut(M_{27})|=54,
\qquad\text{so}\qquad |\Hol(M_{27})|=27\cdot54=1458 .
\]
A direct computation \cite{code} shows that in both cases the stabiliser of
$\sigma$ in $\Hol(G)$ is trivial.  Hence
\[
|\Orb(\sigma)|=11664 \quad\text{for } H_3,
\qquad
|\Orb(\sigma)|=1458 \quad\text{for } M_{27},
\]
so $|\CB(H_3)|\ge 11664$ and $|\CB(M_{27})|\ge 1458$.  (The automorphism
action alone would only give the orbits of size $432$ and $54$.)
\end{rem}

\section{Lifting colouring bijections}\label{sec:lift}

\subsection{Lifting machinery}

We first set up the framework: a colouring bijection of $G$ is assembled
from a colouring bijection of a quotient $G/H$ and a family of bijections
of $H$.

\begin{lem}[Layer criterion]\label{layer}
Let $G$ be a group, $H\trianglelefteq G$, and let
$G=\bigsqcup_{t\in T}Ht$ be the decomposition of $G$ into cosets of $H$
determined by a transversal $T$.  Let $\sigma\colon G\to G$ be a
bijection and let $\bullet\in\{+,-,c\}$.

Assume that for every $t\in T$ there are a bijection
$F^{(t)}_{\bullet}\colon H\to H$ and an element $t_{\bullet}\in T$ with
\[
\Delta^{\,\bullet}_{\sigma}(ht)=F^{(t)}_{\bullet}(h)\;t_{\bullet}
\qquad\text{for all }h\in H .
\]
If $t\mapsto t_{\bullet}$ is a bijection of $T$, then
$\Delta^{\,\bullet}_{\sigma}$ is a bijection of $G$.

In particular, if this holds for all three symbols
$\bullet\in\{+,-,c\}$, then $\sigma\in\CB(G)$.
\end{lem}

\begin{proof}
For a fixed $t\in T$ the map $\Delta^{\,\bullet}_{\sigma}$ sends $Ht$ into
$Ht_{\bullet}$ by $ht\mapsto F^{(t)}_{\bullet}(h)t_{\bullet}$, and this
restriction $Ht\to Ht_{\bullet}$ is a bijection because $F^{(t)}_{\bullet}$
is one.  Since $t\mapsto t_{\bullet}$ is a bijection of $T$, the cosets
$Ht_{\bullet}$ run exactly once through all cosets of $H$.  Hence
$\Delta^{\,\bullet}_{\sigma}$ is a bijection of $G$.  The last assertion is
Definition~\ref{def:scm_bc}.
\end{proof}

\begin{lem}\label{lem:transversal-scheme}
Let $G$ be a group, $H\trianglelefteq G$, let $T$ be a transversal of $H$
in $G$, and let $\theta\colon G/H\to T$ be the section associated with
$\pi$, so that $\pi|_{T}$ and $\theta$ are mutually inverse bijections.

Let $\Phi\in\CB(G/H)$ and put
\[
\varphi\colon T\to T,\qquad \varphi(t)=\theta\bigl(\Phi(\pi(t))\bigr).
\]
For $t\in T$ write $q=\pi(t)$ and
\[
t_{+}=\theta\bigl(q\,\Phi(q)\bigr),
\qquad
t_{-}=\theta\bigl(q^{-1}\Phi(q)\bigr),
\qquad
t_{c}=\theta\bigl(\Phi(q)^{-1}q\,\Phi(q)\bigr).
\]
Then:
\begin{enumerate}
\item $\varphi$ is a bijection of $T$;
\item for every $t\in T$ there exist uniquely determined
$\xi(t),\zeta(t),\omega(t)\in H$ with
\[
t\,\varphi(t)=\xi(t)\,t_{+},
\qquad
t^{-1}\varphi(t)=\zeta(t)\,t_{-},
\qquad
\varphi(t)^{-1}t\,\varphi(t)=\omega(t)\,t_{c};
\]
\item the maps $t\mapsto t_{+}$, $t\mapsto t_{-}$, $t\mapsto t_{c}$ are
bijections of $T$.
\end{enumerate}
\end{lem}

\begin{proof}
(1)  $\varphi=\theta\circ\Phi\circ\pi|_{T}$ is a composition of
bijections.

(2)  Every $g\in G$ is uniquely of the form $g=ht$ with $h\in H$, $t\in T$,
and then $t=\theta(\pi(g))$.  Applying this to the elements
$t\varphi(t)$, $t^{-1}\varphi(t)$ and $\varphi(t)^{-1}t\varphi(t)$, whose
images under $\pi$ are $q\Phi(q)$, $q^{-1}\Phi(q)$ and
$\Phi(q)^{-1}q\Phi(q)$ respectively, we obtain exactly the stated
decompositions, with $t_{+},t_{-},t_{c}$ as defined and with uniquely
determined $H$-parts $\xi(t),\zeta(t),\omega(t)$.

(3)  By construction
$t_{\bullet}=\theta\bigl(\Delta^{\,\bullet}_{\Phi}(\pi(t))\bigr)$ for
$\bullet\in\{+,-,c\}$.  Since $\Phi\in\CB(G/H)$, each
$\Delta^{\,\bullet}_{\Phi}$ is a bijection of $G/H$, and $\theta$, $\pi|_T$
are mutually inverse bijections. Hence each $t\mapsto t_{\bullet}$ is a
bijection of $T$.
\end{proof}

The next lemma is the engine of the whole paper: it reduces the
construction of a colouring bijection of $G$ to the verification of three
explicit conditions inside $H$.

\begin{lem}[Twisted lift]\label{lem:twisted-lift}
Let $G$ be a group, $H\trianglelefteq G$, let $T$, $\theta$, $\Phi$,
$\varphi$, $\xi,\zeta,\omega$ and $t_{+},t_{-},t_{c}$ be as in
Lemma~\ref{lem:transversal-scheme}.

Suppose that for every $t\in T$ we are given a bijection
$\psi_t\colon H\to H$ such that the three maps
\[
\begin{aligned}
P_t(h)&=h\cdot\rho_t\bigl(\psi_t(h)\bigr),\\
M_t(h)&=h^{-1}\cdot\psi_t(h),\\
K_t(h)&=\psi_t(h)^{-1}\cdot h\cdot\rho_t\bigl(\psi_t(h)\bigr)
\end{aligned}
\]
are bijections of $H$.  Then
\[
\sigma\colon G\to G,
\qquad
\sigma(ht)=\psi_t(h)\,\varphi(t)
\qquad (h\in H,\ t\in T)
\]
is a colouring bijection of $G$.  In particular $G$ is colourable.
\end{lem}

\begin{proof}
Since every $g\in G$ is uniquely of the form $ht$, and since each
$\psi_t$ and $\varphi$ are bijections, $\sigma$ is a bijection of $G$.

Let $g=ht$ with $h\in H$, $t\in T$.  Using
$t\,\psi_t(h)=\rho_t(\psi_t(h))\,t$ we obtain
\[
\Dp_\sigma(ht)
=(ht)\,\psi_t(h)\varphi(t)
=h\,\rho_t\bigl(\psi_t(h)\bigr)\,t\,\varphi(t)
=\Bigl(P_t(h)\,\xi(t)\Bigr)\,t_{+},
\]
\[
\begin{aligned}
\Dm_\sigma(ht)
&=(ht)^{-1}\psi_t(h)\varphi(t)
 =t^{-1}\bigl(h^{-1}\psi_t(h)\bigr)\varphi(t)\\
&=\rho_{t^{-1}}\bigl(M_t(h)\bigr)\,t^{-1}\varphi(t)
 =\Bigl(\rho_{t^{-1}}\bigl(M_t(h)\bigr)\,\zeta(t)\Bigr)\,t_{-} .
\end{aligned}
\]
For the third map, first note that
\[
\psi_t(h)^{-1}\,h\,t\,\psi_t(h)
=\psi_t(h)^{-1}\,h\,\rho_t\bigl(\psi_t(h)\bigr)\,t
=K_t(h)\,t ,
\]
whence
\[
\begin{aligned}
\Dc_\sigma(ht)
&=\bigl(\psi_t(h)\varphi(t)\bigr)^{-1}(ht)\bigl(\psi_t(h)\varphi(t)\bigr)
 =\varphi(t)^{-1}K_t(h)\,t\,\varphi(t)\\
&=\rho_{\varphi(t)^{-1}}\bigl(K_t(h)\bigr)\,\varphi(t)^{-1}t\,\varphi(t)
 =\Bigl(\rho_{\varphi(t)^{-1}}\bigl(K_t(h)\bigr)\,\omega(t)\Bigr)\,t_{c}.
\end{aligned}
\]
Thus $\Delta^{\,\bullet}_\sigma(ht)=F^{(t)}_{\bullet}(h)\,t_{\bullet}$ for
$\bullet\in\{+,-,c\}$, where
\[
F^{(t)}_{+}(h)=P_t(h)\,\xi(t),
\qquad
F^{(t)}_{-}(h)=\rho_{t^{-1}}\bigl(M_t(h)\bigr)\,\zeta(t),
\qquad
F^{(t)}_{c}(h)=\rho_{\varphi(t)^{-1}}\bigl(K_t(h)\bigr)\,\omega(t).
\]
As $H\trianglelefteq G$, the maps $\rho_{t^{-1}}$ and
$\rho_{\varphi(t)^{-1}}$ restrict to automorphisms of $H$. Right
multiplications by $\xi(t),\zeta(t),\omega(t)$ are bijections of $H$.
Hence, by hypothesis, all three maps $F^{(t)}_{\bullet}\colon H\to H$ are
bijections.  By Lemma~\ref{lem:transversal-scheme} the maps
$t\mapsto t_{\bullet}$ are bijections of $T$, so Lemma~\ref{layer} shows
that $\Dp_\sigma$, $\Dm_\sigma$ and $\Dc_\sigma$ are bijections of $G$.
Therefore $\sigma\in\CB(G)$.
\end{proof}

The central case is now immediate: there the third condition is vacuous.

\begin{cor}\label{central}
Let $G$ be a group and let $H\le Z(G)$.  If $G/H$ is colourable and $H$
is strongly admissible, then $G$ is colourable.
\end{cor}

\begin{proof}
Choose $\Phi\in\CB(G/H)$ and $\psi\in\SCM(H)$, and apply
Lemma~\ref{lem:twisted-lift} with $\psi_t=\psi$ for all $t\in T$.  Since
$H\le Z(G)$, we have $\rho_t|_H=\mathrm{id}_H$ and $H$ is abelian, so
\[
P(h)=h\,\psi(h)=\Dp_\psi(h),
\qquad
M(h)=h^{-1}\psi(h)=\Dm_\psi(h),
\qquad
K(h)=\psi(h)^{-1}h\,\psi(h)=h .
\]
The first two are bijections because $\psi\in\SCM(H)$, and
$K=\mathrm{id}_H$.  Hence $G$ is colourable.
\end{proof}

\bigskip

\subsection{The two lifting theorems}

In this subsection we prove two lifting theorems: a colouring bijection of
$G/H$ can be lifted to a colouring bijection of $G$ when
$H\cong C_3\times C_3$, and when $H\cong C_9\times C_3$ contains a
central element of order $9$.  In both cases all the work is done by
Lemma~\ref{lem:twisted-lift}: it remains only to produce, for each
$t\in T$, a bijection $\psi_t\colon H\to H$ for which the three maps
\begin{equation}\label{eq:PMK}
P_t(h)=h\cdot\rho_t\bigl(\psi_t(h)\bigr),
\qquad
M_t(h)=h^{-1}\psi_t(h),
\qquad
K_t(h)=\psi_t(h)^{-1}\,h\,\rho_t\bigl(\psi_t(h)\bigr)
\end{equation}
are bijections of $H$.  Only the restriction of $\rho_t$ to $H$ enters
\eqref{eq:PMK}, and by Corollary~\ref{central} we may throughout assume
that $H\not\le Z(G)$.

We begin with a lemma describing this restriction.

\begin{lem}\label{lem:conjugation}
Let $G$ be a $3$-group and suppose
\[
H=\langle c\rangle\times\langle b\rangle\trianglelefteq G,
\qquad c\in Z(G),\qquad |b|=3^{k},\ k\ge1 .
\]
Then for every $t\in G$ there are integers
$m(t)\in\{0,1,\dots,3^{k-1}-1\}$ and $l(t)$ such that
\[
t\,b\,t^{-1}=b^{\,1+3m(t)}c^{\,l(t)} .
\]
\end{lem}

\begin{proof}
Fix $t\in G$.  As $H\trianglelefteq G$ we have $tbt^{-1}\in H$, and since
$c\in Z(G)$, conjugation by $t$ fixes $c$ and induces an automorphism of
$H/\langle c\rangle\cong\langle b\rangle\cong C_{3^{k}}$.  This yields a
homomorphism
$G\to\Aut(C_{3^{k}})\cong(\mathbb Z/3^{k}\mathbb Z)^{\times}$ whose image,
$G$ being a $3$-group, lies in the unique Sylow $3$-subgroup
\[
\{\,1+3m \bmod 3^{k}\ :\ m=0,1,\dots,3^{k-1}-1\,\}
\]
of $(\mathbb Z/3^{k}\mathbb Z)^{\times}$.  Hence
$tbt^{-1}\equiv b^{1+3m(t)}\pmod{\langle c\rangle}$ for some $m(t)$,
which is the assertion.
\end{proof}

\bigskip

\subsection[Case H = C3 x C3]{Case \texorpdfstring{$H=C_3\times C_3$}{H = C3 x C3}}

Here $H$ is regarded as a $2$-dimensional vector space over $\F_3$, and
automorphisms of $H$ are identified with matrices in $\GL(2,\F_3)$ with
respect to a fixed basis. Let $\Mat(f)$ denote the matrix of $f$.  Put
\[
X_1=\begin{pmatrix}0&1\\1&1\end{pmatrix},
\qquad
X_2=\begin{pmatrix}0&1\\2&1\end{pmatrix},
\qquad
C(\lambda)=\begin{pmatrix}1&\lambda\\0&1\end{pmatrix}\quad(\lambda\in\F_3).
\]

\begin{lem}\label{lem:pair-choice}
For every $\lambda\in\F_3$ there is $j\in\{1,2\}$ such that the three
matrices
\[
X_j-I,
\qquad
I+C(\lambda)X_j,
\qquad
I+\bigl(C(\lambda)-I\bigr)X_j
\]
are invertible. One may take $j=1$ for $\lambda\in\{0,1\}$ and $j=2$ for
$\lambda=2$.
\end{lem}

\begin{proof}
Both $X_1,X_2$ lie in $\GL(2,\F_3)$, since $\det X_1=2$ and $\det X_2=1$.
A direct computation gives
\[
\det(X_1-I)=2,
\qquad
\det\bigl(I+C(\lambda)X_1\bigr)=1+\lambda,
\qquad
\det\bigl(I+(C(\lambda)-I)X_1\bigr)=1+\lambda,
\]
and
\[
\det(X_2-I)=1,
\qquad
\det\bigl(I+C(\lambda)X_2\bigr)=2\lambda,
\qquad
\det\bigl(I+(C(\lambda)-I)X_2\bigr)=1+2\lambda .
\]
Hence all three determinants are nonzero for $X_1$ precisely when
$\lambda\neq2$, and for $X_2$ precisely when $\lambda=2$.
\end{proof}

\begin{thm}\label{thm:lift-C3xC3}
Let $G$ be a $3$-group with a normal subgroup $H\cong C_3\times C_3$.
If $G/H$ is colourable, then $G$ is colourable.
\end{thm}

\begin{proof}
The group $C_3\times C_3$ is strongly admissible by
Theorem~\ref{3-group-abelian}, so Corollary~\ref{central} settles the
case $H\le Z(G)$.  Assume therefore $H\not\le Z(G)$.

Since $G$ is a $3$-group and $H\trianglelefteq G$ is nontrivial, we have
$H\cap Z(G)\neq\{e\}$.  As $H\not\le Z(G)$ and $|H|=9$, this forces
$|H\cap Z(G)|=3$.  Choose
$z\in H\cap Z(G)$ of order $3$ and $b\in H\setminus\langle z\rangle$.  As
$H$ has exponent $3$,
\[
H=\langle z\rangle\times\langle b\rangle,
\qquad z\in Z(G),\quad b\notin Z(G).
\]
By Lemma~\ref{lem:conjugation} (applied with $c=z$ and $k=1$), for every
$t\in G$ there is $k(t)\in\{0,1,2\}$ with
\[
\rho_t(z)=z,
\qquad
\rho_t(b)=z^{\,k(t)}b .
\]
Writing every element of $H$ uniquely as $z^{x}b^{y}$ we identify
$H\cong\F_3^{2}$ via $z^{x}b^{y}\leftrightarrow(x,y)$, and in this basis
\begin{equation}\label{eq:rho-matrix}
\Mat(\rho_t)=C\bigl(k(t)\bigr),\qquad t\in G .
\end{equation}

Fix $\Phi\in\CB(G/H)$ and let $T,\varphi,\xi,\zeta,\omega,
t_{+},t_{-},t_{c}$ be as in Lemma~\ref{lem:transversal-scheme}.  For
$t\in T$ set
\[
\lambda(t):=k(t)\in\F_3,
\]
and let $\psi_t\in\Aut(H)$ be the linear automorphism with
\[
\Mat(\psi_t)=X_j,
\qquad
j=j\bigl(\lambda(t)\bigr)\in\{1,2\}
\text{ chosen as in Lemma~\ref{lem:pair-choice}} .
\]

We verify the three conditions of Lemma~\ref{lem:twisted-lift}.  Write
$H$ additively and abbreviate $\Psi=\Mat(\psi_t)$ and
$R=\Mat(\rho_t)=C(\lambda(t))$.  Then \eqref{eq:PMK} reads
\[
P_t(h)=h+R\,\Psi h,
\qquad
M_t(h)=-h+\Psi h,
\qquad
K_t(h)=-\Psi h+h+R\,\Psi h,
\]
so these maps are linear with matrices
\[
\Mat(P_t)=I+R\,\Psi,
\qquad
\Mat(M_t)=\Psi-I,
\qquad
\Mat(K_t)=I+(R-I)\Psi .
\]
By Lemma~\ref{lem:pair-choice} the matrices $I+C(\lambda(t))X_j$,
$X_j-I$ and $I+(C(\lambda(t))-I)X_j$ are invertible, so $P_t$, $M_t$ and
$K_t$ are bijections of $H$.

By Lemma~\ref{lem:twisted-lift}, $\sigma(ht)=\psi_t(h)\varphi(t)$ is a
colouring bijection of $G$. Hence $G$ is colourable.
\end{proof}

\bigskip

\subsection[Case H = C9 x C3]{Case \texorpdfstring{$H=C_9\times C_3$}{H = C9 x C3}}

We now treat a normal subgroup $H\cong C_9\times C_3$ that admits a
decomposition
\[
H=\langle c\rangle\times\langle b\rangle,
\qquad c\in Z(G),\qquad |c|=9,\quad |b|=3 .
\tag{$\dagger$}\label{eq:dagger}
\]
This is exactly the configuration produced by the proof of the main
theorem (Subcase~2b of Theorem~\ref{thm:main}), and it is the only one we
shall need.

Assume \eqref{eq:dagger} and set $z=c^{3}$, so that $|z|=3$ and
$z\in Z(G)$.  By Lemma~\ref{lem:conjugation}, $\rho_t(c)=c$ and
$\rho_t(b)=z^{\,k(t)}b$ with $k(t)\in\{0,1,2\}$, so for
$h=c^{u}b^{j}\in H$,
\[
\rho_t(h)=c^{u}\bigl(z^{k(t)}b\bigr)^{j}=h\cdot z^{\,k(t)j} .
\]
Thus
\begin{equation}\label{eq:rho-Q}
\rho_t(h)=h\cdot p_{\lambda}(h),
\qquad \lambda=k(t),
\qquad
p_{\lambda}\bigl(c^{u}b^{j}\bigr)=z^{\,\lambda j},
\end{equation}
where $p_{\lambda}\colon H\to\langle z\rangle$ is an endomorphism of $H$.
Since $H$ is abelian, substituting \eqref{eq:rho-Q} into \eqref{eq:PMK}
gives, for any bijection $\psi_t\colon H\to H$,
\begin{equation}\label{eq:PMK-abelian}
P_t(h)=h\,\psi_t(h)\,p_{\lambda}\bigl(\psi_t(h)\bigr),
\qquad
M_t(h)=h^{-1}\psi_t(h),
\qquad
K_t(h)=h\,p_{\lambda}\bigl(\psi_t(h)\bigr) .
\end{equation}
Note that $p_{\lambda}$ is evaluated at $\psi_t(h)$, not at $h$, and that
$K_t$ arises from $P_t$ by deleting the factor $\psi_t(h)$.

\begin{lem}\label{alpha}
Let $H=\langle c\rangle\times\langle b\rangle\cong C_9\times C_3$ with
$|c|=9$, $|b|=3$, and $z=c^{3}$.  For every $\lambda\in\{0,1,2\}$ there
exists a bijection $\alpha_{\lambda}\colon H\to H$ such that the three
maps
\[
A_{\lambda}(h)=h\,\alpha_{\lambda}(h)\,p_{\lambda}\bigl(\alpha_{\lambda}(h)\bigr),
\qquad
B_{\lambda}(h)=h^{-1}\alpha_{\lambda}(h),
\qquad
C_{\lambda}(h)=h\,p_{\lambda}\bigl(\alpha_{\lambda}(h)\bigr)
\]
are bijections of $H$.
\end{lem}

\begin{proof}
Identify $H\cong\mathbb Z_9\times\mathbb Z_3$ additively via
$c^{u}b^{j}\leftrightarrow(u,j)$. Then $z=(3,0)$ and
$p_{\lambda}(u,j)=(3\lambda j,\,0)$, and the three maps read
\[
A_{\lambda}(h)=h+\alpha_{\lambda}(h)+p_{\lambda}(\alpha_{\lambda}(h)),
\qquad
B_{\lambda}(h)=-h+\alpha_{\lambda}(h),
\qquad
C_{\lambda}(h)=h+p_{\lambda}(\alpha_{\lambda}(h)).
\]

For $\lambda=0$ we have $p_{0}\equiv0$, so $C_{0}=\mathrm{id}_H$, while the
conditions on $A_{0}$ and $B_{0}$ say precisely that
$\alpha_{0}\in\SCM(H)$.  Such an $\alpha_{0}$ exists by
Theorem~\ref{3-group-abelian}, because $H$ is a noncyclic abelian
$3$-group.

For $\lambda=1,2$ the bijections $\alpha_{\lambda}$ are listed in
Table~\ref{Tab:3-add} of Appendix~\ref{app:C9C3-tables}, together with
$A_{\lambda}$, $B_{\lambda}$ and $C_{\lambda}$. Direct inspection of the
table shows that each of the eight displayed columns is a permutation of
$H$.
\end{proof}

\begin{thm}\label{thm:lift-C9xC3}
Let $G$ be a finite $3$-group and let $H\trianglelefteq G$ satisfy
\eqref{eq:dagger}, that is,
\[
H=\langle c\rangle\times\langle b\rangle\cong C_9\times C_3,
\qquad c\in Z(G),\qquad |c|=9,\quad |b|=3 .
\]
If $G/H$ is colourable, then $G$ is colourable.
\end{thm}

\begin{proof}
Fix $\Phi\in\CB(G/H)$ and let $T,\varphi,\xi,\zeta,\omega,
t_{+},t_{-},t_{c}$ be as in Lemma~\ref{lem:transversal-scheme}.  For
$t\in T$ put $\lambda(t)=k(t)$, as in \eqref{eq:rho-Q}, and
\[
\psi_t=\alpha_{\lambda(t)},
\]
with $\alpha_{\lambda}$ given by Lemma~\ref{alpha}.  Each $\psi_t$ is a
bijection of $H$, and by \eqref{eq:PMK-abelian}
\[
P_t=A_{\lambda(t)},
\qquad
M_t=B_{\lambda(t)},
\qquad
K_t=C_{\lambda(t)},
\]
which are bijections of $H$ by Lemma~\ref{alpha}.  The hypotheses of
Lemma~\ref{lem:twisted-lift} are therefore satisfied, and
$\sigma(ht)=\psi_t(h)\varphi(t)$ is a colouring bijection of $G$.
\end{proof}

\section{Main theorem}\label{sec:Main}

The aim of this section is to prove the main theorem. A crucial role is
played by the following structural result describing $3$-groups with a
cyclic maximal subgroup.

\begin{thm}[{cf.~\cite[Theorem~5.3.4]{R}}]\label{max:cyclic}
Let $G$ be a $3$-group of order $3^r$, $r\ge3$, having a cyclic maximal
subgroup.  Then
\[
G\cong C_{3^{r}},
\qquad
G\cong C_{3^{r-1}}\times C_3
\qquad\text{or}\qquad
G\cong M_{3^{r}} .
\]
In particular, if $G$ is noncyclic, then $G\cong C_{3^{r-1}}\times C_3$ or
$G\cong M_{3^{r}}$.
\end{thm}

The following elementary properties of the groups $M_{3^{k}}$ will be used
repeatedly in the proof of the main theorem.  Throughout, $\Omega_1(G)$
denotes the subgroup $\langle x\in G\mid x^{3}=e\rangle$.

\begin{lem}\label{lem:centre-L}
Let $k\ge3$ and
$M_{3^{k}}=\langle a,b\mid a^{3^{k-1}}=b^{3}=e,\ bab^{-1}=a^{1+3^{k-2}}\rangle$.
Then
\[
Z(M_{3^{k}})=\langle a^{3}\rangle\cong C_{3^{\,k-2}} .
\]
In particular $Z(M_{3^{k}})$ is cyclic, its unique subgroup of order $3$ is
\[
\Omega_1\bigl(Z(M_{3^{k}})\bigr)=\langle a^{3^{\,k-2}}\rangle ,
\]
and
\[
M_{3^{k}}\big/\Omega_1\bigl(Z(M_{3^{k}})\bigr)\cong C_{3^{\,k-2}}\times C_3 .
\]
Moreover
\[
\Omega_1\bigl(M_{3^{k}}\bigr)=\langle a^{3^{\,k-2}}\rangle\times\langle b\rangle
\cong C_3\times C_3 .
\]
\end{lem}

\begin{proof}
Every element of $M_{3^{k}}$ is uniquely of the form $a^{t}b^{s}$ with
$0\le t<3^{k-1}$, $0\le s<3$, and
$b^{s}ab^{-s}=a^{(1+3^{k-2})^{s}}$.  Since
$(1+3^{k-2})^{s}\equiv1+s\,3^{k-2}\pmod{3^{k-1}}$, an element $a^{t}b^{s}$
commutes with $a$ if and only if $3\mid s$, i.e.\ $s=0$. Similarly, $a^{t}$
commutes with $b$ if and only if $t\,3^{k-2}\equiv0\pmod{3^{k-1}}$, i.e.\
$3\mid t$.  Hence $Z(M_{3^{k}})=\langle a^{3}\rangle$, a cyclic group of order
$3^{\,k-2}$, whose unique subgroup of order $3$ is
$\langle a^{3^{\,k-2}}\rangle$.

Put $Q=M_{3^{k}}/\langle a^{3^{\,k-2}}\rangle$, a group of order $3^{\,k-1}$.
In $Q$ the defining relation becomes $\bar b\bar a\bar b^{-1}=\bar a$, so
$Q$ is abelian. Moreover $|\bar a|=3^{\,k-2}$ and $|\bar b|$ divides $3$,
whence $|Q|\le3^{\,k-1}$ forces $|\bar b|=3$ and
$\langle\bar a\rangle\cap\langle\bar b\rangle=\{e\}$.  Therefore
$Q\cong C_{3^{\,k-2}}\times C_3$.

Finally, the defining relation gives $bab^{-1}a^{-1}=a^{3^{\,k-2}}$, an
element of $\langle a^{3}\rangle=Z(M_{3^{k}})$.  As $M_{3^{k}}=\langle
a,b\rangle$ and the commutator of the two generators is central, the derived
subgroup of $M_{3^{k}}$ equals $\langle a^{3^{\,k-2}}\rangle$, of order $3$.
In particular $M_{3^{k}}$ has class $2$, so
$(xy)^{3}=x^{3}y^{3}[y,x]^{3}=x^{3}y^{3}$ and cubing is an endomorphism of
$M_{3^{k}}$.  Its image is generated by $a^{3}$ and
$b^{3}=e$, hence equals $\langle a^{3}\rangle$, of order $3^{\,k-2}$.
Consequently its kernel $\Omega_1(M_{3^{k}})$ has order $9$.  As
$a^{3^{\,k-2}}$ and $b$ are commuting elements of order $3$ generating a
subgroup of order $9$, we get
$\Omega_1(M_{3^{k}})=\langle a^{3^{\,k-2}}\rangle\times\langle b\rangle
\cong C_3\times C_3$.
\end{proof}

The following auxiliary lemma will be used later.  We first isolate the
consequence of Kolchin's theorem on which its proof rests.  Let a $p$-group
$P$ act by automorphisms on an elementary abelian $p$-group $V\neq\{e\}$.
Viewing $V$ as a vector space over $\F_p$ and $P$ as a subgroup of
$\GL(V)$, every $g\in P$ is unipotent: if $g$ has order $p^{m}$, then
$(g-1)^{p^{m}}=g^{p^{m}}-1=0$, since $\operatorname{End}(V)$ has
characteristic $p$.  By Kolchin's theorem \cite{K}, $P$ therefore
stabilises a complete flag
\begin{equation}\label{eq:flag}
\{e\}=V_0<V_1<\cdots<V_n=V,
\qquad
|V_i|=p^{\,i} .
\end{equation}

\begin{lem}\label{lem:normal-pp}
Let $G$ be a noncyclic $3$-group.  Then $G$ contains a normal subgroup
isomorphic to $C_3\times C_3$.
\end{lem}

\begin{proof}
Let $A$ be an abelian normal subgroup of $G$ of largest possible order.

We claim that $C_G(A)=A$.  Since $A$ is abelian, $A\le C:=C_G(A)$, and
$C\trianglelefteq G$ because $A\trianglelefteq G$.  If $A<C$, then $C/A$ is a
nontrivial normal subgroup of the $3$-group $G/A$ and therefore meets
$Z(G/A)$ nontrivially. Choose $x\in C\setminus A$ with $xA\in Z(G/A)$.  Then
$B=\langle A,x\rangle$ is abelian, because $x$ centralises $A$, and
$B\trianglelefteq G$, because $A\trianglelefteq G$ and $x^{g}\in xA\subseteq B$
for every $g\in G$.  As $|B|>|A|$, this contradicts the choice of $A$.

Assume first that $A$ is noncyclic.  Then $\Omega_1(A)$ is elementary
abelian of rank at least $2$ and characteristic in $A$, hence normal in
$G$.  Applying \eqref{eq:flag} to the conjugation action of $G$ on
$\Omega_1(A)$, we obtain a $G$-invariant subgroup $V_2$ of order $9$. It is
normal in $G$ and elementary abelian, so $V_2\cong C_3\times C_3$.

Assume now that $A$ is cyclic.  Then $A\neq G$, since $G$ is noncyclic, so
$G/A$ is a nontrivial $3$-group and $Z(G/A)$ contains a subgroup of order
$3$. Let $H$ be its preimage in $G$, so that $A\le H\trianglelefteq G$ and
$|H:A|=3$.  If $H$ were abelian, then $H\le C_G(A)=A$, which is false.
Hence $H$ is nonabelian and consequently $|H|=3^{r}$ with $r\ge3$.  Thus
$H$ is a nonabelian $3$-group of order $3^{r}$, $r\ge3$, with cyclic
maximal subgroup $A$, so $H\cong M_{3^{r}}$ by Theorem~\ref{max:cyclic}.
By Lemma~\ref{lem:centre-L} the subgroup $\Omega_1(H)\cong C_3\times C_3$ is
characteristic in $H$, and $H\trianglelefteq G$. Hence $\Omega_1(H)$ is normal
in $G$.
\end{proof}

\begin{lem}\label{lem:normal-C3xC3}
Let $G$ be a $3$-group of order $3^{r}>27$, all of whose maximal subgroups
are noncyclic.  Then $G$ contains a normal subgroup $H\cong C_3\times C_3$
such that $G/H$ is not cyclic.
\end{lem}

\begin{proof}
Note first that $r\ge4$ and that $G$ is noncyclic, since a cyclic group
has a cyclic maximal subgroup.  By Lemma~\ref{lem:normal-pp} there is a
normal subgroup $H\trianglelefteq G$ with $H\cong C_3\times C_3$.  If
$G/H$ is noncyclic, we are done, so assume from now on that
\[
G/H \text{ is cyclic.}
\]

Choose $x\in G$ such that $xH$ generates $G/H$. Then $G=\langle x\rangle H$
and
\[
3^{r}=|G|=\frac{|\langle x\rangle|\cdot|H|}{|\langle x\rangle\cap H|}
       =\frac{9\,|\langle x\rangle|}{|\langle x\rangle\cap H|} .
\]
Now $|\langle x\rangle\cap H|=9$ is impossible, since $H$ is not cyclic.  Next,
$|\langle x\rangle\cap H|=3$ would give $|\langle x\rangle|=3^{r-1}$,
i.e.\ a cyclic maximal subgroup of $G$, contrary to the hypothesis.
Hence
\begin{equation}\label{eq:split}
\langle x\rangle\cap H=\{e\},
\qquad
|\langle x\rangle|=3^{\,r-2}\ge 9,
\qquad
G=H\rtimes\langle x\rangle .
\end{equation}

Conjugation by $x$ induces on $H\cong\F_3^{2}$ an automorphism of
$3$-power order, hence a unipotent element $U\in\GL(2,\F_3)$.
Consequently:
\begin{enumerate}
\item[(i)] $U$ has a nonzero fixed vector, so there is $z\in H$ of order
$3$ with $xzx^{-1}=z$;
\item[(ii)] $U^{3}=I$, because $U=I+N$ with $N^{2}=0$ and
$\operatorname{char}\F_3=3$. Hence $x^{3}\in C_G(H)$.
\end{enumerate}

Put
\[
w=x^{\,3^{\,r-3}} .
\]
Since $r\ge4$ we have $w\in\langle x^{3}\rangle\le C_G(H)$, and
$|w|=3$ because $|x|=3^{\,r-2}$ by \eqref{eq:split}.

The element $z$ is centralised by $H$ (which is abelian) and by $x$, and
$G=H\langle x\rangle$, so $z\in Z(G)$.  Likewise $w$ centralises $H$ and
$\langle x\rangle$, so $w\in Z(G)$.  Moreover
$\langle z\rangle\cap\langle w\rangle\le H\cap\langle x\rangle=\{e\}$ by
\eqref{eq:split}.  Therefore
\[
H^{*}:=\langle z\rangle\times\langle w\rangle\cong C_3\times C_3,
\qquad H^{*}\le Z(G),
\]
and in particular $H^{*}\trianglelefteq G$.

It remains to show that $G/H^{*}$ is not cyclic.  If $z^{a}w^{b}\in H$,
then $w^{b}\in H\cap\langle x\rangle=\{e\}$, so $H\cap H^{*}=\langle z\rangle$ and
symmetrically $\langle x\rangle\cap H^{*}=\langle w\rangle$.  Hence
\[
HH^{*}/H^{*}\cong H/\langle z\rangle\cong C_3,
\qquad
\langle x\rangle H^{*}/H^{*}
\cong\langle x\rangle/\langle w\rangle\cong C_{3^{\,r-3}} .
\]
Since $G=H\langle x\rangle$, the group $G/H^{*}$ is the product of these
two subgroups, and
\[
\bigl|G/H^{*}\bigr|=3^{\,r-2}=3\cdot 3^{\,r-3},
\]
so they intersect trivially.  Because $r\ge4$, the subgroup
$\langle x\rangle H^{*}/H^{*}\cong C_{3^{\,r-3}}$ is nontrivial and
therefore contains a subgroup of order $3$, which is distinct from
$HH^{*}/H^{*}$.  Thus $G/H^{*}$ has two distinct subgroups of order $3$
and is not cyclic.  Taking $H^{*}$ in place of $H$ completes the proof.
\end{proof}

\begin{rem}\label{rem:maxcyclic-sharp}
The hypothesis on maximal subgroups cannot be omitted, and for both
exceptional families the conclusion genuinely fails.

For $G=C_{3^{\,r-1}}\times C_3$ ($r\ge3$) the elements of order dividing
$3$ form the subgroup $\Omega_1(G)\cong C_3\times C_3$. Hence this is the
only subgroup of $G$ isomorphic to $C_3\times C_3$, and
$G/\Omega_1(G)\cong C_{3^{\,r-2}}$ is cyclic.

For $G=M_{3^{r}}$ ($r\ge3$) we have
$\Omega_1(G)=\langle a^{3^{\,r-2}}\rangle\times\langle b\rangle\cong C_3\times C_3$
by Lemma~\ref{lem:centre-L}.
Every subgroup isomorphic to $C_3\times C_3$ consists of elements of
order dividing $3$ and therefore equals $\Omega_1(M_{3^{r}})$, while
$M_{3^{r}}/\Omega_1(M_{3^{r}})$ is generated by the image of $a$ and is thus cyclic
of order $3^{\,r-2}$.

By Theorem~\ref{max:cyclic} these two families are, apart from the cyclic
groups, exactly the $3$-groups possessing a cyclic maximal subgroup.
\end{rem}

\medskip

We now prove the main theorem.

\begin{thm}\label{thm:main}
Every noncyclic $3$-group $G$ not isomorphic to $M_{3^{r}}$ for $r\ge4$ is
colourable.
\end{thm}

\begin{proof}
Suppose that the theorem is false, and let $G$ be a counterexample of minimal
order. Write $|G|=3^r$.  Every noncyclic $3$-group of order at most $27$
is colourable: for the abelian ones this is
Theorem~\ref{3-group-abelian}, and the two nonabelian groups of order
$27$ are covered by Lemma~\ref{lem:order27}.  Hence $r\ge4$.

Assume first that $G$ has a cyclic maximal subgroup. Since $G$ is
noncyclic, Theorem~\ref{max:cyclic} gives
\[
G\cong C_{3^{r-1}}\times C_3
\quad\text{or}\quad
G\cong M_{3^{r}} .
\]
In the first case, $G$ is colourable by
Theorem~\ref{3-group-abelian}, while the second contradicts the
hypothesis of the theorem (recall that $r\ge4$). Hence $G$ has no cyclic
maximal subgroup.

By Lemma~\ref{lem:normal-C3xC3}, the group $G$ contains a normal subgroup $N\cong C_3\times C_3$ such that $G/N$ is not cyclic.

The quotient $G/N$ is noncyclic of order $3^{\,r-2}<|G|$, so by the
minimality of $G$ it is either colourable or isomorphic to $M_{3^{k}}$ for some
$k\ge4$. Comparing orders, the only possibility is $k=r-2$.  If $G/N$ is
colourable, then Theorem~\ref{thm:lift-C3xC3} implies that $G$ is
colourable, a contradiction.  Hence
\[
G/N\cong M_{3^{r-2}},\qquad r-2\ge4,
\]
so that $r\ge6$.  (Equivalently: for $r\le5$ the group $G/N$ would be
noncyclic of order at most $27$, hence colourable by the first
paragraph.)

Since $N\trianglelefteq G$ is nontrivial and $G$ is a $3$-group, we have
$N\cap Z(G)\neq\{e\}$.  Choose $z\in N\cap Z(G)$ of order $3$ and
$b\in N\setminus\langle z\rangle$. As $N\cong C_3\times C_3$ has exponent
$3$, this gives
\[
N=\langle z\rangle\times\langle b\rangle,
\qquad z\in Z(G), \qquad |z|=|b|=3.
\]
Let $\bar x,\bar y$ be generators of $G/N\cong M_{3^{r-2}}$ satisfying
\[
\bar x^{3^{r-3}}=\bar y^3=e,
\qquad
\bar y\bar x\bar y^{-1}=\bar x^{1+3^{r-4}}.
\]
Choose lifts $x,y\in G$ with $xN=\bar x$ and $yN=\bar y$. Then
\[
G=\langle x,y,z,b\rangle .
\]
By Lemma~\ref{lem:conjugation}, for each $t\in G$ there exists
$k(t)\in\{0,1,2\}$ such that
\[
tbt^{-1}=z^{k(t)}b.
\]
Since $z$ is central and $|z|=3$, we obtain
\[
b^{-1}tb=z^{k(t)}t,
\qquad
b^{-1}t^3b=t^3,
\]
so
\begin{equation}\label{eq:b-cube}
[b,t^{3}]=1\qquad\text{for all }t\in G .
\end{equation}
Set
\[
c=x^{3^{r-4}}.
\]
Since $\bar x$ has order $3^{\,r-3}$ in $G/N\cong M_{3^{r-2}}$, the element
$cN=\bar x^{\,3^{\,r-4}}$ has order $3$ in $G/N$ and, by
Lemma~\ref{lem:centre-L}, generates $\Omega_1\bigl(Z(G/N)\bigr)$, the
unique subgroup of order $3$ of the centre of $G/N$.  (Note that this says
nothing yet about the order of $c$ itself.  Both $|c|=3$ and $|c|=9$ will
occur below.)  In particular
\begin{equation}\label{eq:c-basic}
c\notin N,\qquad c^{3}\in N,\qquad\text{hence}\qquad c^{9}=1,
\end{equation}
the last equality because $N\cong C_3\times C_3$ has exponent $3$.  Being a
power of $x$, the element $c$ commutes with $x$.  Since $r\ge6$ we may
write $c=\bigl(x^{3^{\,r-5}}\bigr)^{3}$, so $[b,c]=1$ by
\eqref{eq:b-cube}.

Moreover, from
\[
\bar{y}\bar{x}\bar{y}^{-1}=\bar{x}^{1+3^{r-4}}=\bar{c}\,\bar{x}
\]
we obtain
\begin{equation}\label{eq:yxy}
yxy^{-1}=z^{i_{0}} b^{j_{0}} c\, x
\end{equation}
for some $i_{0},j_{0}\in\{0,1,2\}$.
Because $z$ is central, $[b,c]=1$, and $c$ commutes with $x$, we have
\begin{align*}
yx^3y^{-1}
&=(yxy^{-1})^3
=(z^{i_{0}} b^{j_{0}} c x)^3
= z^{3i_{0}}\,(b^{j_{0}} c x)^3 \\
&= (b^{j_{0}} c)(x b^{j_{0}} c x^{-1})(x^2 b^{j_{0}} c x^{-2})x^3.
\end{align*}
Now $x b x^{-1}=z^{k(x)}b$, hence
\[
x^\nu b^{j_{0}} x^{-\nu}=z^{\nu j_{0} k(x)}b^{j_{0}},
\quad  (\nu=0,1,2).
\]
Using also $z^3=1$ and $(b^{j_{0}})^3=1$, we obtain
\begin{equation}\label{eq:yx3y}
yx^3y^{-1}=c^3x^3.
\end{equation}
Since $r\ge6$ we have $3^{\,r-4}\ge9$, so $c^{3^{\,r-4}}=1$ by
\eqref{eq:c-basic}.  Therefore
\[
ycy^{-1}=yx^{3^{r-4}}y^{-1}=(yx^3y^{-1})^{3^{r-5}}= (c^3x^3)^{3^{r-5}}=c^{3^{r-4}}x^{3^{r-4}}=c.
\]
Thus $c$ commutes with $x$ (a power of $x$), with $z$ (which is central),
with $b$ (because $[b,c]=1$) and with $y$ (by the last display). Since
$G=\langle x,y,z,b\rangle$, we conclude
\[
c\in Z(G),\qquad c\neq e,
\]
the latter by \eqref{eq:c-basic}.  Now let
\[
S=\langle c,z\rangle\le Z(G).
\]

\medskip
\noindent
{\bf Case 1.} $|c|=3$.
Since $c\notin N$ and $\langle z\rangle\le N$, we have
$c\notin\langle z\rangle$, so
\[
S=\langle c\rangle\times\langle z\rangle\cong C_3\times C_3 .
\]
From $c^{3}=e$ and \eqref{eq:yx3y} we get
\[
yx^3y^{-1}=x^3,
\]
and since $x^{3}$ commutes with $x$, with $z$ and, by \eqref{eq:b-cube},
with $b$, it follows that $x^3\in Z(G)$.
Put
\[
d=x^{3^{r-5}}.
\]
Since $r\ge6$, we have $d\in\langle x^{3}\rangle\le Z(G)$ and $d^3=c$.
Moreover, for every $t\in G$ we have
\[
tbt^{-1}=z^{k(t)}b\equiv b \pmod S,
\]
so $bS\in Z(G/S)$. Also $dS\in Z(G/S)$ because $d\in Z(G)$.

We claim that $\langle bS\rangle$ and $\langle dS\rangle$ are two distinct
subgroups of order $3$ of $Z(G/S)$.

First, $b\notin S$: otherwise $b=c^{u}z^{v}$ for some $u,v$, and reducing
modulo $N$ gives $c^{u}N=N$, so $3\mid u$ and hence $c^{u}=e$. Thus
$b\in\langle z\rangle$, contradicting
$N=\langle z\rangle\times\langle b\rangle$.  As $b^{3}=e$, the element
$bS$ has order $3$.

Second, $dN$ has order $9$ in $G/N$, because $\bar x$ has order
$3^{\,r-3}$ and $d=x^{3^{\,r-5}}$.  Since $SN/N=\langle cN\rangle$ has
order $3$, this forces $d\notin S$. As $d^{3}=c\in S$, the element $dS$
also has order $3$.

Finally, if $\langle dS\rangle=\langle bS\rangle$, then
$dS=b^\varepsilon S$ for some $\varepsilon\in\{1,2\}$, and therefore
\[
d=b^\varepsilon c^u z^v
\]
for suitable integers $u,v$.  Passing modulo $N=\langle z,b\rangle$, we
obtain $dN=c^uN$, which is impossible because $dN$ has order $9$ whereas
$c^{u}N$ has order at most $3$.  This proves the claim.

Hence $Z(G/S)$ contains two distinct subgroups of order $3$, so $Z(G/S)$
is not cyclic. In particular $G/S$ is not cyclic.  By
Lemma~\ref{lem:centre-L} the centre of $M_{3^{k}}$ is cyclic for every $k\ge3$,
and $|G/S|=3^{\,r-2}$, so $M_{3^{r-2}}$ is the only group $M_{3^{k}}$ of that order.
Consequently
\[
G/S\not\cong M_{3^{k}}\qquad\text{for every }k,
\]
and by the minimality of $G$ the quotient $G/S$ is colourable.  Moreover
$S\cong C_3\times C_3$ is strongly admissible by
Theorem~\ref{3-group-abelian}, and $S\le Z(G)$. Applying
Corollary~\ref{central}, we conclude that $G$ is colourable, a
contradiction.

\medskip
\noindent
{\bf Case 2.} $|c|=9$. Then $c^3\in N$ by \eqref{eq:c-basic}, so we may
write
\[
c^3=z^{i_{1}} b^{j_{1}}
\]
with $i_{1},j_{1}\in\{0,1,2\}$. These exponents are unrelated to
$i_{0},j_{0}$ in \eqref{eq:yxy}.

\smallskip
\noindent
{\it Subcase 2a: $j_{1}\neq 0$.}
Then $b^{j_{1}}=z^{-i_{1}}c^{3}\in\langle c,z\rangle$ and, $j_{1}$ being
invertible modulo $3$, also $b\in\langle c,z\rangle$. Hence
$N=\langle z,b\rangle\le S$.  Furthermore $z\notin\langle c\rangle$:
otherwise $z$, being of order $3$, would lie in $\langle c^{3}\rangle$, so
$z=c^{\pm3}=(z^{i_{1}}b^{j_{1}})^{\pm1}$, which would give
$b^{\pm j_{1}}\in\langle z\rangle$ and hence $b\in\langle z\rangle$, a
contradiction.  Therefore
\[
S=\langle c\rangle\times\langle z\rangle \cong C_9\times C_3,
\qquad |S|=27,
\qquad N \le S\le Z(G).
\]
Moreover, by \eqref{eq:yxy},
\[
yxy^{-1}=z^{i_{0}} b^{j_{0}} c\,x \equiv x \pmod S,
\]
so $G/S$ is abelian.

Since $N\le S$, we have
\[
G/S \cong (G/N)/(S/N).
\]
Now $G/N\cong M_{3^{r-2}}$, and $S/N$ is the subgroup of order $3$ generated by
$cN=x^{3^{r-4}}N$, that is, $S/N=\Omega_1\bigl(Z(G/N)\bigr)$.  Therefore,
by Lemma~\ref{lem:centre-L},
\[
G/S \cong M_{3^{r-2}}\big/\Omega_1\bigl(Z(M_{3^{r-2}})\bigr)
      \cong C_{3^{\,r-4}}\times C_3 .
\]
In particular, $G/S$ is abelian and noncyclic (recall $r\ge6$).  Hence
$G/S$ is colourable by Theorem~\ref{3-group-abelian}.  Moreover
$S\cong C_9\times C_3$ is strongly admissible by
Theorem~\ref{3-group-abelian}, and $S\le Z(G)$. Now Corollary~\ref{central}
implies that $G$ is colourable, a contradiction.

\smallskip
\noindent
{\it Subcase 2b: $j_{1}=0$.}
Then $c^{3}=z^{i_{1}}$ with $i_{1}\in\{1,2\}$, because $|c|=9$. In particular
\begin{equation}\label{eq:z-in-c}
z\in\langle c\rangle
\qquad\text{and}\qquad
\langle c^{3}\rangle=\langle z\rangle .
\end{equation}
Put
\[
H=\langle c,b\rangle .
\]
We check that $H$ satisfies hypothesis \eqref{eq:dagger} of
Theorem~\ref{thm:lift-C9xC3}.

First, $|b|=3$, since $b\in N$ and $N\cong C_3\times C_3$ has exponent
$3$; and $|c|=9$ with $c\in Z(G)$, as established above.  Next,
$\langle c\rangle\cap N$ is a subgroup of the cyclic group
$\langle c\rangle$ of order $9$, and it has exponent at most $3$ because
$N$ has exponent $3$. Hence
\[
\langle c\rangle\cap N\le\langle c^{3}\rangle=\langle z\rangle .
\]
As $b\notin\langle z\rangle$, we get $\langle c\rangle\cap\langle b\rangle=\{e\}$.
Since $[b,c]=1$ and $c\in Z(G)$, the subgroup $H$ is abelian and
\[
H=\langle c\rangle\times\langle b\rangle\cong C_9\times C_3 ,
\qquad c\in Z(G),\quad |c|=9,\quad |b|=3 .
\]
Finally, $H\trianglelefteq G$: by \eqref{eq:z-in-c} we have $N=\langle z,b\rangle\le H$,
and for every $t\in G$
\[
t\,c\,t^{-1}=c,
\qquad
t\,b\,t^{-1}=z^{\,k(t)}b\in N\le H .
\]

It remains to see that $G/H$ is colourable.  Since $N\le H$,
\[
G/H\cong (G/N)\big/(H/N),
\]
and $H/N$ is the subgroup of order $3$ generated by $cN=x^{3^{r-4}}N$,
i.e.\ $H/N=\Omega_1\bigl(Z(G/N)\bigr)$.  Hence, by
Lemma~\ref{lem:centre-L},
\[
G/H\cong M_{3^{r-2}}\big/\Omega_1\bigl(Z(M_{3^{r-2}})\bigr)\cong C_{3^{\,r-4}}\times C_3 ,
\]
which is abelian and noncyclic (recall $r\ge6$), so $G/H$ is colourable
by Theorem~\ref{3-group-abelian}.

By Theorem~\ref{thm:lift-C9xC3}, the group $G$ is colourable --- a
contradiction.

\smallskip
All cases lead to a contradiction, completing the proof.
\end{proof}

\section{\texorpdfstring{$\omega(\Orth(G))$ for groups of exponent $3$}%
{omega(Orth(G)) for groups of exponent 3}}\label{sec:omega}

By Corollary~\ref{cor:exp3} every group $G$ of exponent $3$ with $|G|>3$ is colourable,
and every $\sigma\in\CB(G)$ gives five mutually orthogonal Latin squares
$L_{\mathrm{id}},L_\iota,L_\sigma,L_{\Dp},L_{\Dm}$ of order $|G|$.  Equivalently, the four
vertices $\iota,\sigma,\Dp,\Dm$ form a $K_4$ in $\Orth(G)$.  Hence
\[
\omega\bigl(\Orth(G)\bigr)\ \ge\ 4
\qquad\text{for every group of exponent $3$ with $|G|>3$.}
\]
We show here that for all such groups but one considerably more is true, and we isolate
the single group for which this bound is the best one we know.

The bounds below do not use the Main Theorem itself, only one of its ingredients,
Lemma~\ref{lem:normal-C3xC3}, together with the quotient construction of Quinn \cite{Q}:
\begin{equation}\label{eq:quinn}
\omega\bigl(\Orth(G)\bigr)\ \ge\
\min\Bigl(\omega\bigl(\Orth(H)\bigr),\ \omega\bigl(\Orth(G/H)\bigr)\Bigr)
\qquad\text{for } H\trianglelefteq G,\ \{e\}\neq H\neq G,
\end{equation}
and on the classical value for elementary abelian groups. By \cite[Thm.~25]{E4}, if
$E\cong C_p^{\,s}$ with $s\ge1$, then $\omega(\Orth(E))=p^{s}-2$. A maximum clique
is given by the $p^{s}-2$ linear orthomorphisms $x\mapsto ax$ with
$a\in\mathrm{GF}(p^{s})$, $a\neq0,1$.  Equivalently, the translation squares based on $E$
contain a complete set of $p^{s}-1$ mutually orthogonal Latin squares.

\begin{lem}\label{lem:series-bound}
Let $p$ be a prime and let $G$ be a nontrivial finite $p$-group possessing a chain
\[
\{e\}=N_0<N_1<\dots<N_k=G,
\qquad N_i\trianglelefteq G,\qquad k\ge1,
\]
in which every factor $N_i/N_{i-1}$ is elementary abelian of rank at least $r\ge1$.
Then
\[
\omega\bigl(\Orth(G)\bigr)\ \ge\ p^{\,r}-2 .
\]
\end{lem}

\begin{proof}
Induction on $k$.  If $k=1$ then $G=N_1$ is elementary abelian of some rank $s\ge r$, so
$\omega(\Orth(G))=p^{s}-2\ge p^{r}-2$.

Let $k\ge2$ and put $H=N_1$, which is normal in $G$ and elementary abelian of rank
$s\ge r$, whence $\omega(\Orth(H))=p^{s}-2\ge p^{r}-2$.  The subgroups $N_i/H$ for
$1\le i\le k$ are normal in $G/H$, form a chain from the trivial subgroup to $G/H$, and
their successive factors are $(N_i/H)\big/(N_{i-1}/H)\cong N_i/N_{i-1}$, again elementary
abelian of rank at least $r$.  By the inductive hypothesis
$\omega(\Orth(G/H))\ge p^{r}-2$, and \eqref{eq:quinn} gives
$\omega(\Orth(G))\ge\min(p^{s}-2,\,p^{r}-2)=p^{r}-2$.
\end{proof}

\begin{cor}\label{cor:exp3-even}
Let $G$ be a group of exponent $3$ with $|G|=3^{2n}$, $n\ge1$.  Then $G$ has a chain of
normal subgroups with all factors isomorphic to $C_3\times C_3$, and consequently
\[
\omega\bigl(\Orth(G)\bigr)\ \ge\ 7 .
\]
\end{cor}

\begin{proof}
We construct the chain by induction on $n$.  For $n=1$ the group $G$ has order $9$ and
exponent $3$, hence $G\cong C_3\times C_3$ and the chain $\{e\}\le G$ will do.

Let $n\ge2$, so that $|G|=3^{2n}\ge81>27$.  Every element of $G$ has order at most $3$,
while a cyclic maximal subgroup would have order $3^{2n-1}\ge27$ and would therefore
contain an element of order greater than $3$. Hence all maximal subgroups of $G$ are
noncyclic and Lemma~\ref{lem:normal-C3xC3} applies.  It yields a normal subgroup
$N\trianglelefteq G$ with $N\cong C_3\times C_3$.  The quotient $G/N$ has exponent $3$,
being a quotient of $G$, and order $3^{2n-2}$, so by the inductive hypothesis it has a
chain of normal subgroups with all factors isomorphic to $C_3\times C_3$.  Taking
preimages in $G$ and prefixing $\{e\}\le N$ produces the required chain for $G$, since
preimages of subgroups normal in $G/N$ are normal in $G$ and the corresponding factors are
isomorphic.  The bound now follows from Lemma~\ref{lem:series-bound} with $p=3$ and $r=2$.
\end{proof}

The restriction to even powers of $3$ is only apparent. The odd ones reduce to it in a
single step, once the groups of order $27$ are set aside.

\begin{cor}\label{cor:exp3-all}
Let $G$ be a group of exponent $3$ with $|G|>27$.  Then
\[
\omega\bigl(\Orth(G)\bigr)\ \ge\ 7 .
\]
\end{cor}

\begin{proof}
Write $|G|=3^{n}$, so that $n\ge4$.  For $n$ even this is
Corollary~\ref{cor:exp3-even}, so assume $n$ is odd, $n\ge5$.

We first observe that $G$ cannot be generated by two elements.  Indeed, a $2$-generator
group of exponent $3$ is a quotient of the free Burnside group $B(2,3)$, which is finite
of order $27$ by Burnside's theorem \cite[\S14.2]{R}. As $|G|>27$, at least three
generators are needed.  Since $\exp(G)=3$ we have $\Phi(G)=G'$, and the elementary
abelian group $G/\Phi(G)$ therefore has rank $d\ge3$.

Choose a subgroup $H$ with $\Phi(G)\le H\le G$ whose image in $G/\Phi(G)$ has
codimension $3$. This is possible because $d\ge3$.  Every subgroup containing $\Phi(G)$
is normal in $G$, the quotient $G/\Phi(G)$ being abelian, so $H\trianglelefteq G$ and
$G/H\cong C_3^{\,3}$, whence $\omega\bigl(\Orth(G/H)\bigr)=3^{3}-2=25$.  On the other
hand $H$ has exponent $3$ and order $3^{\,n-3}$ with $n-3$ even and positive, so
Corollary~\ref{cor:exp3-even} gives $\omega\bigl(\Orth(H)\bigr)\ge7$.  Now
\eqref{eq:quinn} yields $\omega\bigl(\Orth(G)\bigr)\ge\min(7,25)=7$.
\end{proof}

\begin{cor}\label{cor:extraspecial-omega}
Let $G$ be an extraspecial group of order $3^{1+2m}$, $m\ge1$, of either exponent $3$ or
exponent $9$.  Then
\[
\omega\bigl(\Orth(G)\bigr)\ \ge\ 3^{m}-2 .
\]
\end{cor}

\begin{proof}
Write $Z=Z(G)=G'$, a group of order $3$, which we identify once and for all with the
additive group of $\F_3$, and put $V=G/Z$, an $\F_3$-vector space of dimension
$2m$.  The commutator induces a map
\[
\beta\colon V\times V\to Z\cong\F_3,\qquad \beta(xZ,yZ)=[x,y],
\]
which is well defined and $\F_3$-bilinear because $G$ has class $2$ and $G'=Z$ has
exponent $3$, and which is alternating since $[x,x]=e$.  It is nondegenerate: if
$\beta(xZ,\cdot)=0$ then $x$ centralises $G$, that is $x\in Z$.

Since $G$ has class $2$, one has $(xy)^{3}=x^{3}y^{3}[y,x]^{\binom{3}{2}}=x^{3}y^{3}$,
the last equality because  $\exp(G')=3$.  Hence $x\mapsto x^{3}$ is a
homomorphism $G\to Z$, and it kills $G'=Z$, so it induces a linear form
\[
\varphi\colon V\to Z\cong\F_3,\qquad \varphi(xZ)=x^{3},
\]
which vanishes identically if $\exp(G)=3$ and is nonzero if $\exp(G)=9$.  By
nondegeneracy of $\beta$ there is a unique $u\in V$ with $\varphi=\beta(u,\cdot)$, so that
$\ker\varphi=u^{\perp}$. In the exponent $3$ case $u=0$.

Choose a totally isotropic subspace $W\le V$ of dimension $m$ containing $u$. Such a
subspace exists, since $u$ spans a totally isotropic subspace and every totally isotropic
subspace of a nondegenerate alternating form is contained in one of dimension $m$.  As
$W$ is totally isotropic we have $W^{\perp}=W$, whence $u\in W$ gives
$W\subseteq u^{\perp}=\ker\varphi$.

Let $H\le G$ be the preimage of $W$, so that $|H|=3^{m+1}$.  It is abelian, because
$[x,y]=\beta(xZ,yZ)=0$ for $x,y\in H$, and of exponent $3$, because
$x^{3}=\varphi(xZ)=0$, that is $x^{3}=e$, for $x\in H$. Hence $H\cong C_3^{\,m+1}$.  Moreover $H\supseteq Z=G'$,
so $H$ is normal in $G$ and $G/H\cong V/W$ is elementary abelian of order $3^{m}$.
Lemma~\ref{lem:series-bound} applied to the chain $\{e\}\le H\le G$, whose factors are
elementary abelian of ranks $m+1$ and $m$, gives $\omega(\Orth(G))\ge3^{m}-2$.
\end{proof}

\begin{rem}\label{rem:exp3-sharpness}
Corollary~\ref{cor:exp3-all} improves the bound $\omega(\Orth(G))\ge4$ of
Corollary~\ref{cor:exp3} to $\omega(\Orth(G))\ge7$ for every group of exponent $3$ of
order greater than $27$.  For the extraspecial groups of exponent $3$ and order
$3^{1+2m}$, Corollary~\ref{cor:extraspecial-omega} improves this further as soon as
$m\ge3$, where $3^{m}-2\ge25$. For $m=2$ the two bounds agree.
Among the groups of exponent $3$ with $|G|>3$ this leaves only the two of order $27$.
For $C_3\times C_3\times C_3$ the value $\omega(\Orth(G))=25$ is classical
\cite[Thm.~25]{E4}.  For $H_3$ we know nothing beyond $\omega(\Orth(H_3))\ge4$.  Indeed
$m=1$ there, so Corollary~\ref{cor:extraspecial-omega} returns only the trivial bound
$1$; and every nontrivial proper normal subgroup of $H_3$ is either $Z(H_3)\cong C_3$ or
a subgroup of index $3$ with quotient $C_3$, so that, $\omega(\Orth(C_3))$ being $1$, the
right-hand side of \eqref{eq:quinn} equals $1$ for every admissible choice of $H$.  Hence
$H_3$ is the unique group of exponent $3$ with $|G|>3$ for which Corollary~\ref{cor:exp3}
remains the best bound available to us.
\end{rem}

\begin{rem}\label{rem:exp9-extraspecial}
For the extraspecial groups of exponent $9$ the comparison with
Corollary~\ref{cor:exp3} does not arise, that corollary not applying to them.  For
$m\ge2$ Corollary~\ref{cor:extraspecial-omega} is then the only bound we know beyond
$\omega(\Orth(G))\ge2$, which holds for every colourable group by
Corollary~\ref{cor:mols-3groups}.  For $m=1$, that is for $M_{27}$, it returns $1$ and so
says less than that, and $\omega(\Orth(M_{27}))\ge2$ is all we have.
\end{rem}

\section{Concluding remarks}

\begin{enumerate}
\item
Every finite nilpotent group is a direct product of its Sylow subgroups. Hence,
if $G$ has odd order and its Sylow $3$-subgroup is noncyclic and not
isomorphic to any $M_{3^{r}}$ ($r\ge 4$), then $G$ is colourable. In particular:

\begin{cor}
Let $G$ be a nilpotent group of odd order whose Sylow $3$-subgroup is
neither cyclic of order greater than $1$ nor isomorphic to any group
$M_{3^{r}}$ ($r\ge 4$). Then
\[
\chi(\mathscr{G}_3(G)) = |G|.
\]
\end{cor}

\item
The colourability of the groups $M_{3^{r}}$ for $r\ge 4$ remains open.  

\item  If $G$ and $H$ are nontrivial $3$-groups, then the direct product $G\times H$ is colourable.
Indeed, its centre satisfies
\[
Z(G\times H)=Z(G)\times Z(H),
\]
which is noncyclic as a direct product of nontrivial abelian $3$-groups.  Hence $G\times H$
is itself noncyclic, and it is not isomorphic to any group $M_{3^{s}}$, whose centre is
cyclic by Lemma~\ref{lem:centre-L}.  Theorem~\ref{thm:main} therefore applies.

\item
By Theorem~\ref{thm:main}, every noncyclic $3$-group that is not isomorphic to $M_{3^{r}}$ ($r\ge4$) is colourable.
A natural direction for further study is the case of $2$-groups. Recent results of Akhtar, Charboneau and Gagola \cite{ACG} on strong complete mappings 
for $2$-groups suggest that an analogous phenomenon may occur in this setting. In particular, it is of interest to determine 
which strongly admissible $2$-groups are colourable.

\item
Another natural direction is to study the colourability of solvable groups, particularly those of order $2^m3^n$.

\end{enumerate}

\section*{Data availability}

The code verifying all computational statements of this paper is openly available
at \cite{code} under the MIT licence.  It is written in Python, has no third-party
dependencies, and reproduces both tables of the appendix.

\section*{Declaration of generative AI and AI-assisted technologies in the writing
process}

During the preparation of this work the author used Claude (Anthropic) in order to
check the language and editorial consistency of the manuscript, and to assist in
preparing the verification code and its documentation cited as \cite{code}.  After
using this tool, the author reviewed and edited the content as needed and takes full
responsibility for the content of the publication.

\appendix

\section{\texorpdfstring{Bijections for $\mathbb{Z}_9\times\mathbb{Z}_3$}{Bijections for Z9 x Z3}}
\label{app:C9C3-tables}

All computational statements of this paper were obtained by exhaustive computer search:
the colouring bijections of $H_3$ and $M_{27}$ in Table~\ref{tab:H3-L3-add-fixing-id}, the
triviality of the holomorph stabilisers in Remark~\ref{rem:27-orbits}, the enumeration of
$\CB(M_{16})$ in Remark~\ref{rem:scm-not-cb}, the pair of matrices of
Lemma~\ref{lem:pair-choice}, and the bijections $\alpha_1,\alpha_2$ below.
The entries of Tables~\ref{tab:H3-L3-add-fixing-id} and~\ref{Tab:3-add} were in addition
verified cell by cell against the group laws: for every element the printed values of
$\Dp$, $\Dm$, $\Dc$ (respectively $A_\lambda$, $B_\lambda$, $C_\lambda$) were recomputed,
and each column was confirmed to be a permutation.  The code that carries out all of this
is written in Python, requires no third-party libraries, and is archived at
\cite{code}; it also rediscovers each of the objects above by an independent search that
does not consult the tables.

Elements of $H\cong C_9\times C_3$ are written additively as
pairs $(u,j)\in\mathbb Z_9\times\mathbb Z_3$ corresponding to $c^{u}b^{j}$,
where $|c|=9$, $|b|=3$ and $z=c^{3}=(3,0)$. Recall that
$p_{\lambda}(u,j)=(3\lambda j,\,0)$.  For $\lambda=0$ one has
$p_{0}\equiv 0$, so no table is needed: the conditions of
Lemma~\ref{alpha} reduce to $\alpha_{0}\in\SCM(H)$, and such a map exists
by Theorem~\ref{3-group-abelian}.

\begin{longtable}{c cccc cccc}
\caption{The bijections $\alpha_\lambda$ of Lemma~\ref{alpha} for $\lambda=1,2$, in the additive
coordinates $(u,j)\leftrightarrow c^{u}b^{j}$ of $H\cong\mathbb Z_9\times\mathbb Z_3$,
together with $A_\lambda(h)=h+\alpha_\lambda(h)+p_\lambda(\alpha_\lambda(h))$,
$B_\lambda(h)=-h+\alpha_\lambda(h)$ and $C_\lambda(h)=h+p_\lambda(\alpha_\lambda(h))$.
Each column is a permutation of $H$.}
\label{Tab:3-add}\\
\toprule
& \multicolumn{4}{c}{$\lambda=1$} & \multicolumn{4}{c}{$\lambda=2$}\\
\cmidrule(lr){2-5}\cmidrule(lr){6-9}
$h$ & $\alpha_1(h)$ & $A_1(h)$ & $B_1(h)$ & $C_1(h)$
    & $\alpha_2(h)$ & $A_2(h)$ & $B_2(h)$ & $C_2(h)$\\
\midrule
\endfirsthead
\multicolumn{9}{c}{\textit{Table \ref{Tab:3-add} continued}}\\
\toprule
& \multicolumn{4}{c}{$\lambda=1$} & \multicolumn{4}{c}{$\lambda=2$}\\
\cmidrule(lr){2-5}\cmidrule(lr){6-9}
$h$ & $\alpha_1(h)$ & $A_1(h)$ & $B_1(h)$ & $C_1(h)$
    & $\alpha_2(h)$ & $A_2(h)$ & $B_2(h)$ & $C_2(h)$\\
\midrule
\endhead
\bottomrule
\endfoot
$(0,0)$ & $(0,0)$ & $(0,0)$ & $(0,0)$ & $(0,0)$ & $(0,0)$ & $(0,0)$ & $(0,0)$ & $(0,0)$\\
$(0,1)$ & $(1,0)$ & $(1,1)$ & $(1,2)$ & $(0,1)$ & $(5,2)$ & $(8,0)$ & $(5,1)$ & $(3,1)$\\
$(0,2)$ & $(7,0)$ & $(7,2)$ & $(7,1)$ & $(0,2)$ & $(5,1)$ & $(2,0)$ & $(5,2)$ & $(6,2)$\\
$(1,0)$ & $(3,1)$ & $(7,1)$ & $(2,1)$ & $(4,0)$ & $(0,2)$ & $(4,2)$ & $(8,2)$ & $(4,0)$\\
$(1,1)$ & $(6,1)$ & $(1,2)$ & $(5,0)$ & $(4,1)$ & $(8,1)$ & $(6,2)$ & $(7,0)$ & $(7,1)$\\
$(1,2)$ & $(0,1)$ & $(4,0)$ & $(8,2)$ & $(4,2)$ & $(4,2)$ & $(8,1)$ & $(3,0)$ & $(4,2)$\\
$(2,0)$ & $(1,1)$ & $(6,1)$ & $(8,1)$ & $(5,0)$ & $(6,1)$ & $(5,1)$ & $(4,1)$ & $(8,0)$\\
$(2,1)$ & $(4,1)$ & $(0,2)$ & $(2,0)$ & $(5,1)$ & $(4,0)$ & $(6,1)$ & $(2,2)$ & $(2,1)$\\
$(2,2)$ & $(5,1)$ & $(1,0)$ & $(3,2)$ & $(5,2)$ & $(1,0)$ & $(3,2)$ & $(8,1)$ & $(2,2)$\\
$(3,0)$ & $(7,1)$ & $(4,1)$ & $(4,1)$ & $(6,0)$ & $(2,0)$ & $(5,0)$ & $(8,0)$ & $(3,0)$\\
$(3,1)$ & $(2,1)$ & $(8,2)$ & $(8,0)$ & $(6,1)$ & $(1,2)$ & $(7,0)$ & $(7,1)$ & $(6,1)$\\
$(3,2)$ & $(8,1)$ & $(5,0)$ & $(5,2)$ & $(6,2)$ & $(1,1)$ & $(1,0)$ & $(7,2)$ & $(0,2)$\\
$(4,0)$ & $(4,2)$ & $(5,2)$ & $(0,2)$ & $(1,0)$ & $(7,2)$ & $(5,2)$ & $(3,2)$ & $(7,0)$\\
$(4,1)$ & $(7,2)$ & $(8,0)$ & $(3,1)$ & $(1,1)$ & $(0,1)$ & $(1,2)$ & $(5,0)$ & $(1,1)$\\
$(4,2)$ & $(1,2)$ & $(2,1)$ & $(6,0)$ & $(1,2)$ & $(6,2)$ & $(4,1)$ & $(2,0)$ & $(7,2)$\\
$(5,0)$ & $(2,2)$ & $(4,2)$ & $(6,2)$ & $(2,0)$ & $(7,1)$ & $(0,1)$ & $(2,1)$ & $(2,0)$\\
$(5,1)$ & $(5,2)$ & $(7,0)$ & $(0,1)$ & $(2,1)$ & $(5,0)$ & $(1,1)$ & $(0,2)$ & $(5,1)$\\
$(5,2)$ & $(6,2)$ & $(8,1)$ & $(1,0)$ & $(2,2)$ & $(6,0)$ & $(2,2)$ & $(1,1)$ & $(5,2)$\\
$(6,0)$ & $(8,2)$ & $(2,2)$ & $(2,2)$ & $(3,0)$ & $(7,0)$ & $(4,0)$ & $(1,0)$ & $(6,0)$\\
$(6,1)$ & $(3,2)$ & $(6,0)$ & $(6,1)$ & $(3,1)$ & $(3,2)$ & $(3,0)$ & $(6,1)$ & $(0,1)$\\
$(6,2)$ & $(0,2)$ & $(3,1)$ & $(3,0)$ & $(3,2)$ & $(3,1)$ & $(6,0)$ & $(6,2)$ & $(3,2)$\\
$(7,0)$ & $(5,0)$ & $(3,0)$ & $(7,0)$ & $(7,0)$ & $(8,2)$ & $(0,2)$ & $(1,2)$ & $(1,0)$\\
$(7,1)$ & $(2,0)$ & $(0,1)$ & $(4,2)$ & $(7,1)$ & $(4,1)$ & $(8,2)$ & $(6,0)$ & $(4,1)$\\
$(7,2)$ & $(8,0)$ & $(6,2)$ & $(1,1)$ & $(7,2)$ & $(2,2)$ & $(3,1)$ & $(4,0)$ & $(1,2)$\\
$(8,0)$ & $(3,0)$ & $(2,0)$ & $(4,0)$ & $(8,0)$ & $(2,1)$ & $(7,1)$ & $(3,1)$ & $(5,0)$\\
$(8,1)$ & $(6,0)$ & $(5,1)$ & $(7,2)$ & $(8,1)$ & $(3,0)$ & $(2,1)$ & $(4,2)$ & $(8,1)$\\
$(8,2)$ & $(4,0)$ & $(3,2)$ & $(5,1)$ & $(8,2)$ & $(8,0)$ & $(7,2)$ & $(0,1)$ & $(8,2)$\\
\end{longtable}

\end{document}